\documentclass[a4paper]{amsart}
\usepackage{amssymb,amsmath,amsthm,graphicx}
\theoremstyle{plain}
\newtheorem{theorem}{Theorem}[section]

\newtheorem{lemma}[theorem]{Lemma}
\newtheorem{corollary}[theorem]{Corollary}
\newtheorem{proposition}[theorem]{Proposition}
\newtheorem{assumption}[theorem]{Assumption}
\theoremstyle{remark}
\newtheorem{remark}[theorem]{Remark}

\usepackage
[dvipdfm,
   bookmarks=true,
   bookmarksnumbered=false,
   bookmarkstype=toc]
{hyperref}

\numberwithin{equation}{section}

\newcommand{\abs}[1]{\left\lvert #1\right\rvert}
\newcommand{\norm}[1]{\left\lVert #1\right\rVert}
\newcommand{\Lebn}[2]{\left\lVert #1 \right\rVert_{L^{#2}}}

\newcommand{\Jbr}[1]{\left\langle #1 \right\rangle}

\begin{document}
\title
{Analysis of Hartree equation with an interaction
growing at the spatial infinity.
}
\author
{Satoshi Masaki}
\date{}
\maketitle
\vskip-5mm
\centerline{Laboratory of Mathematics, 
Institute of Engineering, Hiroshima University,}
\centerline{Higashihiroshima Hiroshima, 739-8527, Japan}
\centerline{masaki@amath.hiroshima-u.ac.jp}

\begin{abstract}
We consider nonlinear Schr\"odinger equation with a Hartree-type nonlocal nonlinearity.
The case where a nonlinear interaction potential grows
at the spatial infinity is studied.
By virtue of an effective decomposition of the nonlinearity based on
conservation of mass, this kind of growing nonlinear interaction is known to contain an effect like a linear potential.
In this paper, 
a well-posedness result is obtained in a suitable energy space
for a class of interaction potential growing at the spatial infinity in at most the quadratic order.
If the growth rate of interaction potential is faster than the linear order,
a priori information on the center of mass plays a crucial role.
When the interaction potential is exactly in the quadratic order, solution are written explicitly.
\end{abstract}

\section{Introduction}

This article is devoted to the study of the Cauchy problem of Hartree equation
\begin{equation}\label{eq:H}\tag{H}
	\left\{
	\begin{aligned}
	&i \partial_t u + \frac12 \Delta u = \eta (|x|^{-\nu}*|u|^2) u, \\
	&u(0)=u_0,
	\end{aligned}
	\right.
\end{equation}
where $(t,x) \in {\mathbb{R}}\times {\mathbb{R}}^d$, $d\geqslant1$, and $\eta\in{\mathbb{R}}$.
The function space to which the initial data $u_0$ belong will be specified later. 
We treat the case where the exponent $\nu$ is negative.
More specifically, let us consider $\nu \in [-2,0)$.
In such a case, the nonlinear interaction potential $\eta |x|^{-\nu}$ grows at the spatial infinity.
Growing nonlocal interaction naturally appears in physical contexts.
Indeed, Hartree equation \eqref{eq:H} is a generalized models of 
Schr\"odinger-Poisson system
\begin{equation}\label{eq:SP}\tag{SP}
	\left\{
	\begin{aligned}
	&i \partial_t u + \frac12 \Delta u = V_{\mathrm{P}} u, \\
	&-\Delta V_{\mathrm{P}} = |u|^2, \\
	&u(0)=u_0,
	\end{aligned}
	\right.
\end{equation}
which is a model equation for semiconductors.
When $d\geqslant 3$ the potential $V_{\mathrm{P}}$ is given by
	$V_{\mathrm{P}}(x)=c_d (|x|^{-(d-2)}*|u|^2)$,
where $c_d$ is a positive constant.
Hence \eqref{eq:SP} corresponds to the special case of
\eqref{eq:H} such that $\eta= c_d>0$
and $\nu = d-2 >0$.
We would remark that, in this case, nonlinear potential decays at the spatial infinity.
On the other hand, when $d\leqslant2$ the potential $V_{\mathrm{P}}$ has completely different shape:
\[
	V_{\mathrm{P}}(x) = \left\{
	\begin{aligned}
	&-\frac1{2\pi}(\log|x|*|u|^2), \quad (d=2), \\
	& -\frac1{2}(|x|*|u|^2), \quad (d=1).
	\end{aligned}
	\right.
\]
One sees that that the nonlinearity grows at the spatial infinity, which are the cases we want to work with.
In particular, when $d=1$ \eqref{eq:SP} corresponds to 
\eqref{eq:H} with $\eta=-1/2<0$ and $\nu=-1<0$.

This article is a consequence of \cite{Ma2DSPe} in which
global well-posedness of \eqref{eq:SP} for dimensions $d=1$ and $d=2$
is shown in an energy class
(see also \cite{DR-JMP,StrSIAMMA,StiMMMAS} for one dimensional case).
Another type of local existence result on \eqref{eq:SP} for $d=2$
is established in \cite{Ma2DSP}.
Since the global well-posedness of \eqref{eq:H} for $\nu \in [-1,0)$ 
follows by an adaption of the argument in \cite{Ma2DSPe} (see Theorem \ref{thm:appendix1}),
in this article, we concentrate on 
$\nu \in [-2,-1)$, in which case
 the growth rate of the nonlinear potential is faster
than in the previous results.
It will turn out that a priori information on the center of mass plays a crucial role.
Further, when $\nu=-2$, the equation \eqref{eq:H} is solved explicitly.

To make notation clear, we introduce $\gamma=-\nu\in (0,2]$
and $\lambda=-\eta$, and consider
\begin{equation}\label{eq:nH}\tag{nH}
	\left\{
	\begin{aligned}
	&i \partial_t u + \frac12 \Delta u = -\lambda (|x|^{\gamma}*|u|^2) u, \\
	&u(0)=u_0.
	\end{aligned}
	\right.
\end{equation}
In what follows, we call \eqref{eq:nH} as a \emph{negative Hartree equation}
and distinguish it from \eqref{eq:H} by assuming $\nu,\gamma>0$.
We use the notation $\Jbr{x}:=(1+|x|^2)^{1/2}$ for $x\in {\mathbb{R}}^d$
and denote by ${\mathcal{F}}$ the Fourier transform in ${\mathbb{R}}^d$;
\[	
	{\mathcal{F}} f (\xi)=(2\pi)^{-d/2}\int_{{\mathbb{R}}^d}e^{-ix\cdot\xi}f(x)dx.
\]
For nonnegative numbers $s$ and $r$, we define a function space $\Sigma^{s,r}$ by
\[
	\Sigma^{s,r}= \Sigma^{s,r}({\mathbb{R}}^d) := \{f \in H^s({\mathbb{R}}^d) | \Jbr{x}^r f(x) \in L^2({\mathbb{R}}) \}
\]
with a norm $\norm{f}_{\Sigma^{s,r}}^2:= \norm{f}_{H^s({\mathbb{R}}^d)}^2+ \norm{\Jbr{x}^{r}f}_{L^2({\mathbb{R}}^d)}^2$, where $H^s$ stands for the usual Sobolev space.
Let us introduce an \emph{energy}
\[
	E[u(t)] = \frac12 \Lebn{\nabla u(t)}2^2
	- \frac{\lambda}{4} \iint_{{\mathbb{R}}^d\times {\mathbb{R}}^d} |x-y|^{\gamma}
	|u(t,x)|^2|u(t,y)|^2 dxdy,
\]
a \emph{center of mass}
\[
	X[u(t)] = \int_{{\mathbb{R}}^d} y |u(t,y)|^2 dy,
\]
and a \emph{momentum}
\[
	P[u(t)] = \operatorname{Im} \int_{{\mathbb{R}}^d} \overline{u(t,y)} \nabla u(t,y) dy
	=\int_{{\mathbb{R}}^d} \xi |{\mathcal{F}} u(t,\xi)|^2 d\xi.
\]
Notice that $E[u]$, $X[v]$, and $P[w]$ make sense for $u \in \Sigma^{1,\gamma/2}$,
$v\in \Sigma^{0,1/2}$, and $w\in H^{1/2}$, respectively.

\subsection{\texorpdfstring{Main result 1 - the case $\gamma<2$}{Main result 1}}
We first sate our result for $\gamma\in (1,2)$.

\begin{theorem}\label{thm:main1}
	Let $d\geqslant 1$, $\gamma \in (1,2)$, and $\lambda \in {\mathbb{R}}$.
	Then, \eqref{eq:nH} is globally well-posed in $\Sigma^{1,\gamma/2}$.
	More precisely,
	for any $u_0 \in \Sigma^{1,\gamma/2}$, there exists a global solution
	$u\in C({\mathbb{R}};\Sigma^{1,\gamma/2}) \cap C^1({\mathbb{R}};(\Sigma^{1,\gamma/2})^{\prime})$ to \eqref{eq:nH}.
	The solution conserves the mass $\Lebn{u(t)}2$, the energy $E[u(t)]$, and 
	the momentum $P[u(t)]$.
	The solution is unique in the following class:
\[
	\left\{  u\in C({\mathbb{R}},\Sigma^{1,\gamma/2}) \Bigm| P[u(t)] = \mathrm{const.} \right\}.
\]	
\end{theorem}
\begin{remark}
\begin{enumerate}
\item 
We say that uniqueness holds unconditionally if it holds
under $C({\mathbb{R}};\Sigma^{1,\gamma/2})$.
In this theorem, the uniqueness of \eqref{eq:nH} holds conditionally since
the conservation of momentum is additionally required.
For the unconditional uniqueness of nonlinear Schr\"odinger equation
with power nonlinearity, we refer the reader to \cite{FPT-CM,FT-CCM,KaJAM95,WT-HMJ}.
\item Conservation of $P[u(t)]$ is equivalent to $X[u(t)]\equiv M ({\bf a}t+{\bf b})$, where
$M$, ${\bf a}$, and ${\bf b}$ are defined in \eqref{def:M} and \eqref{def:ab} below.
\item We have the following bound on $L^2$-norms of $\nabla u(t)$ and $\Jbr{x}^{\gamma/2}u(t)$
(see also Remark \ref{rmk:expgrow}):
\begin{equation}\label{eq:dispdu}
	\Lebn{\nabla u(t)}2 \leqslant \left\{
	\begin{aligned}
	&C\Jbr{t}^{\frac{\gamma}{2-\gamma}} && \text{if }\lambda>0,\\
	&C && \text{if }\lambda<0, 
	\end{aligned}
	\right.
\end{equation}
\begin{equation}\label{eq:dispxu}
	\Lebn{\Jbr{x}^{\frac{\gamma}2} u(t)}2 \leqslant \left\{
	\begin{aligned}
	&C\Jbr{t}^{\frac{\gamma}{2-\gamma}} && \text{if }\lambda>0,\\
	&C\Jbr{t}^{\frac{\gamma}2} && \text{if }\lambda<0.
	\end{aligned}
	\right.
\end{equation}
\item We show this theorem in a general framework
by replacing $\lambda |x|^\gamma$ with an abstract potential
in a class of divergent functions
(Theorem \ref{thm:general}).
\end{enumerate}
\end{remark}
\begin{remark}
A formula $|x|^{-d+\alpha} * = \Lambda(\alpha) (-\Delta)^{-\frac\alpha2}$
is known for $0<\alpha<d$,
where $\Lambda(\alpha)=2^{\alpha-d/2}\Gamma(\alpha/2)/\Gamma(d/2-\alpha/2)$
(see \cite{StBook}).
If $\gamma \in (0,2)$ and if $\gamma-2>-d$ then we have
\[
	-\Delta (|x|^\gamma*|u|^2) = -\gamma(\gamma-2+d) (|x|^{\gamma-2}*|u|^2)
	=\Lambda(d+\gamma)(-\Delta)^{-\frac{d+\gamma}{2}+1}|u|^2.
\]
In a sense, this implies $|x|^\gamma*=\Lambda(d+\gamma) (-\Delta)^{-\frac{d+\gamma}2}$ ,
which is an extension of the above formula to $\alpha=d+\gamma>d$.
\end{remark}
\subsection{\texorpdfstring{Main Result 2 - the case $\gamma=2$}{Main result 2}}
We next consider the case $\gamma=2$.
There exists an explicit solution in this case.
Moreover, uniqueness property holds without conservation of momentum.

Assume $u_0 \in \Sigma^{1,1}$ and introduce the following vectors and numbers.
We first let 
\begin{equation}\label{def:M}
M= \norm{u_0}_{L^2}^2
\end{equation}
 and define constant vectors
\begin{align}
	{\bf a} &{}:= M^{-1}P[u_0],&
	{\bf b} &{}:= M^{-1}X[u_0],
	\label{def:ab}
\end{align}
which represents the (scaled) momentum and
the center of mass, respectively. Let
\begin{equation}\label{eq:cde}
\begin{aligned}
	c&{}:=\Lebn{\nabla u_0}2^2, &
	d&{}:= \operatorname{Im}\int \overline{u_0} x\cdot \nabla u_0 dx  , &
	e&{}:=\Lebn{x u_0 }2^2.
\end{aligned}
\end{equation}
For $\omega \in {\mathbb{R}}$, we set 
$\mathcal{U}_{\omega}(t) = e^{it(\frac{\Delta}{2}+ \frac{\omega}{2}|x|^2)}$.
An integral representation of $\mathcal{U}_{\omega}(t)$ is known 
as Mehler's formula.
For a vector ${\bf a}$, 
let $\tau_{\bf a}$ and $\pi_{\bf a}$ be translation operators defined by
$(\tau_{\bf a}f)(x)=f(x-{\bf a})$ and $(\pi_{\bf a} f)(x)=e^{ix\cdot {\bf a}}f(x)$,
respectively.

\begin{theorem}\label{thm:main2}
Let $d\geqslant1$, $\gamma=2$ and $\lambda=\pm1/2$. 
Then, \eqref{eq:nH} is globally well-posed in $\Sigma^{1,1}$.
Moreover, the uniqueness holds unconditionally.
Furthermore, the unique solution of \eqref{eq:nH} is given by
\[
	u(t) = \left\{
	\begin{aligned}
	&\exp\left({i\frac{|{\bf a}|^2}2t + i\psi_+(t)}\right) \tau_{{\bf a}t+{\bf b}}\pi_{{\bf a}}\mathcal{U}_{M}(t) \pi_{-{\bf a}}
	\tau_{-{\bf b}} u_0, &&\text{if }\lambda=1/2, \\
	&\exp\left({i\frac{|{\bf a}|^2}2t + i\psi_-(t)}\right) \tau_{{\bf a}t+{\bf b}}\pi_{{\bf a}}
 \mathcal{U}_{-M}(t) \pi_{-{\bf a}}
	\tau_{-{\bf b}} u_0, &&\text{if }\lambda=-1/2,
	\end{aligned}
	\right.
\]
where
\begin{equation*}
\begin{aligned}
	\psi_+(t) ={}&
	\frac{c -M|{\bf a}|^2 +M(e-M|{\bf b}|^2) }{4M^{3/2}}\sinh (\sqrt{M}t)\cosh (\sqrt{M}t) 
	\\&{} 
	+ \frac{d-M{\bf a}\cdot{\bf b}}{2M}\sinh^2 (\sqrt{M}t) 
 + \frac{- c + M|{\bf a}|^2 +M(e - M|{\bf b}|^2)}{4M} t 
\end{aligned}
\end{equation*}
and
\[
\begin{aligned}
	\psi_-(t) ={}& \frac{c-M|{\bf a}|^2 -M (e- M|{\bf b}|^2)}{4M^{3/2}}\sin (\sqrt{M}t)\cos (\sqrt{M}t) \\
	&{} + \frac{M{\bf a}\cdot{\bf b} - d}{2M} \sin^2 (\sqrt{M}t) 
	+ \frac{- c+M|{\bf a}|^2 -M(e-M|{\bf b}|^2)}{4M} t .
\end{aligned}
\]
\end{theorem}
By a scaling argument, the general $\lambda>0$ case and $\lambda<0$ case are reduced to the
case $\lambda=1/2$ and the case $\lambda=-1/2$, respectively. 
\begin{remark}
From the explicit representation of the solution, we can deduce that
the nonlinearity causes the following three effects on large time behavior
of the solution.
First is the nonlinear phase $e^{i\psi_{\pm}(t)}$.
It is known that if $\nu>1$ (if $\gamma<-1$) then the nonlinear dynamics
is compared with free one 
but if $\nu\leqslant 1$ (if $\gamma \geqslant -1$) then it is not and a phase correction
must be taken into account (cf. long-range scattering \cite{GO-CMP}).
It seems that $e^{i\psi_{\pm}(t)}$ is a correction of this kind.
Furthermore, the linear dynamics of which the nonlinear dynamics can be regarded 
as a perturbation changes from $e^{it\Delta/2}$ into $e^{it(\Delta\pm M|x|^2)/2}$.
This is the second respect.
It is important to remark that this modified linear dynamics depends
on the mass of the solution.
This phenomena occurs at least for $\gamma>0$.
When $\gamma=0$, a solution of \eqref{eq:nH} is given by
$u(t)=e^{-i\lambda t M} e^{it\Delta/2}u_0$, which is a free dynamics with a phase
correction.
Third is the translation in both Fourier and physical spaces,
which depends only on the momentum ${\bf a}$ and the center of mass ${\bf b}$.
This represents the motion of the center of mass and is involved at least for $\gamma>1$.
\end{remark}
\begin{remark}\label{rmk:expgrow}
Let $\gamma=2$ and $\lambda=1/2$.
A calculation shows
\begin{multline*}
	\Lebn{ \nabla u }2^2=\Lebn{(\cosh(\sqrt{M}t)\nabla + i\sqrt{M} \sinh(\sqrt{M}t)x) u_0}2^2 \\
	+M\abs{\bf a}^2 - M\abs{ \cosh(Mt){\bf a}+ \sqrt{M}\sinh (\sqrt{M}t){\bf b} }^2
\end{multline*}
and
\begin{multline*}
	\Lebn{ \sqrt{M}x u }2^2=\Lebn{\left({\sinh(\sqrt{M}t)}\nabla + i\sqrt{M}\cosh(\sqrt{M}t)x\right) u_0}2^2 \\
	+M^2\abs{{\bf a}t+{\bf b}}^2 - M\abs{ \sinh(Mt){\bf a}+ \sqrt{M}\cosh (\sqrt{M}t){\bf b} }^2.
\end{multline*}
Hence, $\Lebn{ \nabla u (t)}2=O(e^{|t|})$ and $\Lebn{ x u (t)}2=O(e^{|t|})$ as $|t|\to\infty$
for a suitable data (for example, $u_0(x)=e^{-|x|^2}$).
These growth rates are much faster than those for free solutions; $\Lebn{\nabla e^{i{t\Delta}/2}u_0}2=O(1)$
and $\Lebn{x e^{i{t\Delta}/2}u_0}2=O(|t|)$ as $|t|\to\infty$.
This is because the nonlinearity, which is regarded as a repulsive quadratic potential and a remainder,
accelerates the dispersion.
Carles studies effects of repulsive quadratic potentials in \cite{CaSIAMMA}.
The above exponential growths for $\gamma=2$ are, in a sense, equalities of \eqref{eq:dispdu}
and \eqref{eq:dispxu} in the limit $\gamma\uparrow2$.
If a similar acceleration occurred for $\gamma<2$ and $\lambda>0$,
it seems reasonable that the time
growths of $\Lebn{\nabla u}2$ and $\Lebn{xu}2$ are faster than those of free solutions
as in \eqref{eq:dispdu} and \eqref{eq:dispxu}.
\end{remark}

\subsection{\texorpdfstring{Transformation of \eqref{eq:nH}}{Transformation of (nH)}}
What is difficult when we solve \eqref{eq:nH} is the fact that the nonlinear potential 
$(|x|^\gamma*|u|^2)$ grows at the spatial infinity.
For this, it is hard to apply a usual perturbation argument to the corresponding integral equation.
To overcome this respect, we introduce a transformation of \eqref{eq:nH}, which is a key ingredient
of our argument.
Let us now observe this with a formal computation.

We consider the case $\gamma=2$ as a model. 
The equation is then
\[
	i\partial_t u + \frac12 \Delta u = -\lambda (|x|^2*|u|^2)u.
\]
The right hand side is equal to
\[
	-\lambda |x|^2 \Lebn{u(t)}2^2 u(x)
	 + 2\lambda x \cdot X[u] u(x)
	-\lambda \int_{{\mathbb{R}}^d} |y|^2|u(y)|^2 dy u(x).
\]
As long as $\lambda \in{\mathbb{R}}$, we can expect that $\Lebn{u(t)}2$ is conserved.
Hence the first term is regarded as $-\lambda M |x|^2 u(x)$, where $M$ is 
as in \eqref{def:M}.
Now, $u$ solves
\begin{equation}\label{eq:aux1}
	i\partial_t u + \frac12 \Delta u + \lambda M |x|^2u
	=  2\lambda x \cdot X[u]u
	-\lambda \int_{{\mathbb{R}}^d} |y|^2|u(y)|^2 dy u.
\end{equation}
Although the right hand side of this equation is still divergent,
the main part of the nonlinearity is removed and so
the growth rate is not $O(|x|^2)$ any longer but $O(|x|^1)$ as $|x|\to\infty$.
This argument is introduced in \cite{Ma2DSPe}.
Now, let us go one step further.
We next observe from \eqref{eq:nH} that 
\[
	\frac{d}{dt}X[u(t)] = P[u(t)]
\]
follows. Similarly, by a formal calculation, one verifies that
$\frac{d}{dt}P[u(t)]=0$.
Thus, integrating twice gives us
\[
	X[u(t)] = M({\bf a}t + {\bf b}),
\]
where ${\bf a}$ and ${\bf b}$ are defined in \eqref{def:ab}.
Now, we introduce a new unknown
\begin{equation}\label{def:ut}
	\widetilde{u}(t,x) = e^{-i \frac{|{\bf a}|^2}{2}t}
	(\pi_{{-\bf a}} \tau_{-{\bf a}t-{\bf b}} u)(x)  .
\end{equation}
Namely, we work with the center-of-mass frame.
Then, one verifies that $\widetilde{u}$ also solves \eqref{eq:aux1} and
$X[\widetilde{u}(t)] \equiv 0$.
These facts imply that $\widetilde{u}$ is a solution to
\begin{equation}\label{eq:aux2}
	i\partial_t \widetilde{u} + \frac12 \Delta \widetilde{u} + \lambda M |x|^2\widetilde{u}
	=  -\lambda \int_{{\mathbb{R}}^d} |y|^2|\widetilde{u}(y)|^2 dy \widetilde{u}.
\end{equation}
Now, the right hand side is bounded with respect to $x$.
Let us further set
\begin{equation}\label{def:w}
	w(t,x)= \widetilde{u}(t,x) \exp\left(-i \lambda \int_0^t \int_{{\mathbb{R}}^2} |y|^2 |\widetilde{u}(s,y)|^2 dyds\right).
\end{equation}
Then, $w$ solves a linear Schr\"odinger equation $i\partial_t w + \frac12 \Delta w + \lambda M|x|^2 w=0$.
Applying inverses of \eqref{def:w} and \eqref{def:ut},
we obtain an explicit solution of \eqref{eq:nH}.

The argument in the case $\gamma \in (1,2)$ is similar.
We introduce $\widetilde{u}$ as in \eqref{def:ut} and try to solve
a \emph{modified Hartree equation}
\begin{equation}\label{eq:mH}\tag{mH}
\left\{
\begin{aligned}
	&i \partial_t \widetilde{u} + \frac12 \Delta \widetilde{u} + \lambda M |x|^{\gamma}\chi(|x|) \widetilde{u} \\
	&\qquad\qquad {}= -\lambda \int_{{\mathbb{R}}^d} \left(|x-y|^{\gamma}-|x|^\gamma\chi(|x|) + \gamma \Jbr{x}^{\gamma-2}x\cdot y\right)|\widetilde{u}(y)|^2dy \widetilde{u}, \\
	& \widetilde{u}(0)=\pi_{-\bf a}\tau_{-\bf b} u_0
\end{aligned}
\right.
\end{equation}
instead of \eqref{eq:aux2},
where $\chi$ is a smooth non-decreasing function such that
$\chi(r)=0$ for $r\leqslant 1$ and $\chi(r)=1$ for $r\geqslant2$.
It will turn out that
\eqref{eq:mH} can be solved in a standard way because
the growth of the nonlinearity of \eqref{eq:nH} is
successfully removed by the transformation.
It is important to note that \eqref{eq:nH} is \emph{not} a perturbation of
free equation $i\partial_t \psi + (1/2)\Delta \psi=0$ any longer but
of $i\partial_t \psi + (1/2)\Delta \psi + \lambda M |x|^{\gamma}\chi(|x|) \psi=0$,
which involves a linear potential.

Oh considered in \cite{OhJDE} the Cauchy problem of 
nonlinear Schr\"odinger equation with a divergent potential and
$L^2$-subcritical power-type nonlinearity
(see also \cite{CazBook}).
In particular, the case where the potential is a quadratic polynomial
is extensively studied.
We refer the reader to \cite{CaAHP,CaSIAMMA,CaDCDS,CMS-SIAM,KVZ-CPDE,WZ-NA,ZxFM}.
 
The rest of this article is organized as follows:
We prove Theorem \ref{thm:main2} in the next Section,
and Theorem \ref{thm:main1} in Section 3.

\section{\texorpdfstring{Proof of Theorem \ref{thm:main2}}{Proof of Theorem 1.4}}
Let us prove our theorem for an equation with a harmonic potential
\begin{equation}\label{eq:2H}
	i\partial_t u + \frac12 \Delta u + \frac{\eta}2 |x|^2 u = -\frac\zeta2 (|x|^2*|u|^2)u, \quad
	u(0) = u_0\in \Sigma^{1,1}({\mathbb{R}}^d),
\end{equation}
where $\eta$ and $\zeta$ are real constants.
For $\omega\in {\mathbb{R}}$ and ${\bf a},{\bf b}\in {\mathbb{R}}^d$, we define an ${\mathbb{R}}^d$-valued
function $g_\omega(t)$ as
\begin{equation}\label{def:gomega}
	g_\omega (t) =
	\left\{
	\begin{aligned}
	&{\bf a}\frac{\sinh (\sqrt\omega t)}{\sqrt{\omega}} + 
	{\bf b}\cosh (\sqrt\omega t) , &&\omega>0, \\
	&{\bf a}t + {\bf b}, &&\omega=0 ,\\
	&{\bf a}\frac{\sin (\sqrt{|\omega|} t)}{\sqrt{|\omega|}}  + 
	{\bf b}\cos (\sqrt{|\omega|} t), &&\omega<0.
	\end{aligned}
	\right.
\end{equation}
Notice that $g_\omega(t)$ is a solution to $g_\omega^{\prime\prime}(t)=\omega g_\omega(t)$
with $g(0)={\bf b}$ and $g^\prime(0)={\bf a}$.
\begin{theorem}\label{thm:harmonic}
\begin{enumerate}
\item Let $d\geqslant1$, $\eta\in {\mathbb{R}}$, and $\zeta\in {\mathbb{R}}$.
Then, \eqref{eq:2H} is globally well-posed in $\Sigma^{1,1}$.
The uniqueness holds unconditionally.
Moreover, $X[u(t)]=M g_\eta(t)$ holds.
\item For a data $u_0 \in \Sigma^{1,1}({\mathbb{R}}^d)$,
define $M$ by \eqref{def:M} and set $\omega = \eta + \zeta M$.
Let $g_\iota(t)$ be defined in \eqref{def:gomega} with a parameter $\iota \in {\mathbb{R}}$
and the data ${\bf a}$ and ${\bf b}$ given by \eqref{def:ab}.
Then, the unique solution to \eqref{eq:2H} is written as
\[
	u(t,x) = e^{i\Psi_{\eta,\zeta}(t)}[ \tau_{g_\eta(t)} \pi_{g_\eta^\prime(t)}
	\pi_{-g_\omega^\prime(t)} \tau_{-g_\omega(t)}  \mathcal{U}_\omega(t)u_0](x)
\]
with
\[
	\Psi_{\eta,\zeta}(t) = \frac12(g_\eta(t)\cdot g_\eta^\prime(t) - g_\omega(t)\cdot
	g_\omega^\prime(t))
	-\frac{\zeta M}2\int_0^t |g_\omega(s)|^2 ds + \frac\zeta2 \psi_\omega(t),
\]
where $\psi_\omega$ is defined with $c$, $d$, and $e$ given by \eqref{eq:cde}
as follows:
\begin{itemize}
\item If $\omega>0$ then
\begin{align*}
	\psi_\omega(t) ={}&
	\frac{c+\omega e}{2\omega^{3/2}}\sinh(\sqrt{\omega}t) \cosh(\sqrt{\omega}t)
	+\frac{d}{\omega}\sinh^2(\sqrt{\omega}t) -\frac{c-\omega e}{2\omega}t
	;
\end{align*}
\item if $\omega=0$ then
\[
	\psi_0(t) = \frac13 ct^3 + dt^2 + et;
\]
\item if $\omega<0$ then
\begin{align*}
	\psi_\omega(t) ={}& -\frac{ 
	c +\omega e
	}{2|\omega|^{3/2}}\sin(\sqrt{|\omega|}t) \cos(\sqrt{|\omega|}t) 
	+\frac{d}{|\omega|}\sin^2(\sqrt{|\omega|}t)
	+\frac{
	c -\omega e
	}{2|\omega|}t.
\end{align*}
\end{itemize}
\end{enumerate}
\end{theorem}
\begin{remark}
Thanks to Proposition \ref{prop:translation} below,
Theorem \ref{thm:main2} immediately follows by taking
$\eta=0$ and $\zeta=\pm1$ (and so $\omega=\pm M$).
\end{remark}
\begin{remark}
In general, momentum of a solution is not conserved in the presence of a linear potential.
Indeed, the solution of \eqref{eq:2H} given in this theorem
satisfies $P[u(t)]=M g^\prime_\eta (t)$.
This is not conserved unless $\eta=0$ (or $u_0$ satisfies ${\bf a}={\bf b}=0$).
\end{remark}
\begin{remark}
Up to a translation in both physical and Fourier spaces and a
nonlinear phase, the solution behaves as $\mathcal{U}_\omega(t)u_0$.
In particular, we have $\Lebn{u(t)}p=\Lebn{\mathcal{U}_\omega(t)u_0}p$ for all $p$.
It is worth pointing out that
not $\eta$ but the exponent $\omega = \eta + \zeta M$ decides the linear profile
of the solution.
This means that, from the view point of change of the dispersive property,
the nonlinearity has the same effect as by the linear potential.
Recall that, however, the motion of the center of mass $X[u(t)]$ is 
governed only by the linear potential.
An interesting case would be $\eta\zeta<0$.
In this case, there exists a critical mass $M_c = -\eta/\zeta$
such that the sign of $\omega$ changes at this value.
If $\Lebn{u_0}2^2=M_c$ then 
the effect of the linear potential is partially removed by the nonlinearity so that
the solution of \eqref{eq:2H} is
a solution of the free Schr\"odinger equation $e^{i\frac{\Delta}2t}u_0$
up to a translation and a nonlinear phase.
\end{remark}
\begin{proof}
Let us first consider the equation
\begin{equation}\label{eq:aux4eta}
	i \partial_t w + \frac12 \Delta w + \frac{\omega}2 |x|^2 w= 0,
	\quad w(0) 
=\pi_{-{\bf a}} \tau_{-{\bf b}} u_0.
\end{equation}
Obviously, a solution is given by $w(t)=\mathcal{U}_{\omega}(t)\pi_{-{\bf a}} \tau_{-{\bf b}} u_0$.
One verifies that $\frac{d}{dt}X[w(t)] = P[w(t)]$ and
\[
	\frac{d}{dt}P[w(t)] = \omega X[w(t)].
\]
Since $X[w(0)] = \int (y-{\bf b})|u_0|^2 dy =0$
and $P[w(0)]=P[u_0]-M {\bf a}=0$ hold,
it follows that $X[w(t)] \equiv 0$.
Set
\begin{equation}\label{def:iw}
	\widetilde{u}(t,x) = w(t,x) \exp \left( 
	i\frac\zeta2 \int_0^t \int_{{\mathbb{R}}^d} |y|^2|w(s,y)|^2 dy ds
	\right).
\end{equation}
Now, it is easy to see that $\widetilde{u}$ solves
\begin{equation}\label{eq:aux3eta}
	i \partial_t \widetilde{u} + \frac12 \Delta \widetilde{u}
	+ \frac\omega2 |x|^2 \widetilde{u}
	= -\frac\zeta2 \int_{{\mathbb{R}}^d} |y|^2|\widetilde{u}(y)|^2 dy \widetilde{u}
\end{equation}
and $\widetilde{u}(0)=w(0)=\pi_{-{\bf a}} \tau_{-{\bf b}} u_0$.
Since $|\widetilde{u}(t,x)|=|w(t,x)|$, it also holds that 
$X[\widetilde{u}(t)]\equiv X[w(t)]\equiv0$.
We introduce $g_\eta(t)$ by \eqref{def:gomega} with 
the data ${\bf a}$ and ${\bf b}$ given in \eqref{def:ab}.
Recall that $g_\eta^{\prime\prime} = \eta g_\eta(t)$.
Hence, 
\begin{multline}\label{eq:aux2eta}
	i \partial_t \widetilde{u} + \frac12 \Delta \widetilde{u}
	+ \frac\omega2 |x|^2 \widetilde{u} +(\eta g_\eta(t)-g_\eta^{\prime\prime}(t))\cdot x \widetilde{u} \\
	= \zeta x \cdot X[\widetilde{u}] \widetilde{u}
	-\frac\zeta2 \int_{{\mathbb{R}}^d} |y|^2|\widetilde{u}(y)|^2 dy \widetilde{u}.
\end{multline}
Since $\omega = \eta + \zeta M = \eta + \zeta \Lebn{\widetilde{u}(t)}2^2$,
this equation is equivalent to
\[
	i \partial_t \widetilde{u} + \frac12 \Delta \widetilde{u}
	+ \frac\eta2 |x+g_\eta(t)|^2 \widetilde{u}
	-g_\eta^{\prime\prime}(t)\cdot x \widetilde{u} -\frac{\eta}2 \abs{g_\eta(t)}^2 \widetilde{u}
	= -\frac\zeta2 \widetilde{u} \int |x-y|^2 |\widetilde{u}(y)|^2 dy .
\]
Hence, if we define $u$ by
\begin{equation}\label{def:iut}
	{u}(t,x) = e^{\frac{i}2\int_0^t (|g^\prime_\eta(s)|^2 + \eta |g_\eta(s)|^2 )ds}
	( \tau_{g_\eta(t)}\pi_{g^\prime_\eta(t)} \widetilde{u})(x),
\end{equation}
then $u$ solves \eqref{eq:2H}.
It also holds that $u(0)=\tau_{{\bf b}}\pi_{{\bf a}} \widetilde{u}(0)=
\tau_{{\bf b}}\pi_{{\bf a}} \pi_{-{\bf a}} \tau_{-{\bf b}} u_0=u_0$.
Combining \eqref{def:iut} and \eqref{def:iw}, we conclude that
\begin{align}
	u(t,x)={}&
	e^{i \frac{1}{2}\int_0^t (|g^\prime_\eta(s)|^2 +\eta |g_\eta(s)|^2)ds}
	( \tau_{g_\eta(t)} \pi_{g^\prime_\eta(t)} \widetilde{u})(x)\nonumber\\
	={}&
	e^{i \widetilde{\Psi}_{\eta,\zeta}(t)}
	( \tau_{g_\eta(t)} \pi_{g^\prime_\eta(t)}
	\mathcal{U}_\omega(t)\pi_{-{\bf a}} \tau_{-{\bf b}} u_0)(x), \label{eq:explicitsol}
\end{align}
where
\begin{align*}
	\widetilde{\Psi}_{\eta,\zeta}(t)={}& \frac{1}2\int_0^t (|g^\prime_\eta(s)^2| +\eta |g_\eta(s)|^2) ds + \frac\zeta2 \int_0^t \int_{{\mathbb{R}}^d} |y|^2|w(s,y)|^2 dy ds\\
	={}& \frac12(g_\eta(t)\cdot g_\eta^\prime(t) - {\bf a}\cdot{\bf b})+ \frac\zeta2 \int_0^t \int_{{\mathbb{R}}^d} |y|^2|w(s,y)|^2 dy ds
\end{align*}
Then, Propositions \ref{prop:translation} and \ref{prop:phaseformula} show the stated
representation of the solution.

Now, let us proceed to the proof of the uniqueness.
Suppose that $u_1\in C({\mathbb{R}};\Sigma^{1,1})$ is a solution of \eqref{eq:2H}
in $(\Sigma^{1,1})^\prime$ sense.
By the equation, we see that $u_1 \in C^1 ({\mathbb{R}};(\Sigma^{1,1})^\prime)$ and so that
$\Lebn{u_1}2$ is conserved and
$\frac{d}{dt} X[u_1(t)] = P[u_1(t)] \in C({\mathbb{R}})$ holds.
By Proposition \ref{prop:motion} below, we obtain $P[u_1(t)] \in C^1({\mathbb{R}})$ and $\frac{d}{dt}P[u_1(t)]=\eta X[u_1(t)]$.
Then, $X[u_1(t)] = M g_\eta(t)$.
Let us introduce
\begin{equation*}
	\widetilde{u}_1(t,x) = 
	e^{-\frac{i}2\int_0^t (|g^\prime_\eta(s)|^2+\eta|g_\eta(s)|^2)ds}
	(\pi_{-g^\prime_\eta(t)} \tau_{-g_\eta(t)}u_1)(x) ,
\end{equation*}
which is the inverse transform of \eqref{def:iut}.
Then, $\widetilde{u}_1$ solves \eqref{eq:aux2eta}. Since
\[
	X[\widetilde{u}_1(t)] = \int (y-g_\eta(t))|u_1(t,y)|^2 dy
=X[u_1(t)] - M g_\eta(t)=0,
\]
$\widetilde{u}$ is also a solution to \eqref{eq:aux3eta}.
Now, let us further introduce
\begin{equation*}
	w_1(t,x) = \widetilde{u}_1(t,x) \exp \left( 
	-i\frac\zeta2 \int_0^t \int_{{\mathbb{R}}^d} |y|^2|\widetilde{u}_1(s,y)|^2 dy ds
	\right).
\end{equation*}
This is the inverse transform of \eqref{def:iw} since
$|w_1(t,x)|=|\widetilde{u}_1(t,x)|$.
Then, $w_1$ solves \eqref{eq:aux4eta} and so $w_1(t,x)=w(t,x)$.
Applying \eqref{def:iut} and \eqref{def:iw}, we conclude that
$u_1$ is identical to $u$. 
\end{proof}

\begin{proposition}\label{prop:translation}
Let $\kappa\in {\mathbb{R}}$ and ${\bf a}$, ${\bf b} \in {\mathbb{R}}^d$.
Define $g_\kappa(t)$ by \eqref{def:gomega}.
Then, it holds for all $t\in {\mathbb{R}}$ that
\[
	\mathcal{U}_{\kappa}(t) \pi_{-{\bf a}} \tau_{-{\bf b}}
	= 
	e^{-\frac{i}2(g_\kappa(t)\cdot g_\kappa^\prime(t) - {\bf a}\cdot{\bf b})}
	\pi_{-g_\kappa^\prime(t)} \tau_{-g_\kappa(t)} \mathcal{U}_{\kappa}(t).
\]
\end{proposition}
\begin{proof}
Fix $\varphi \in \Sigma^{1,1}$.
Let us set
$w(t)=\mathcal{U}_{\kappa}(t) \pi_{-{\bf a}} \tau_{-{\bf b}} \varphi$.
Then,
\[
	i\partial_t w + \frac12 \Delta w + \frac\kappa2 |x|^2 w = 0,\quad
	w(0) = \pi_{-\bf a} \tau_{-\bf b} \varphi.
\]
Now, we introduce
$
	\widetilde{w}(t,x) = e^{\frac{i}2(g_\kappa(t)\cdot g_\kappa^\prime(t) - {\bf a}\cdot{\bf b})}
	[\tau_{g_\kappa(t)} \pi_{g_\kappa^\prime(t)}  w(t)](x).
$
One easily verifies that
$i\partial_t \widetilde{w} + \frac12 \Delta \widetilde{w} + \frac\kappa2 |x|^2 \widetilde{w} = 0$
and $\widetilde{w}(0) = \varphi$ hold,
which implies $\widetilde{w}(t)=\mathcal{U}_\kappa(t) \varphi$.
Hence,
\[
	\mathcal{U}_{\kappa}(t) \varphi
	= e^{\frac{i}2(g_\kappa(t)\cdot g_\kappa^\prime(t) - {\bf a}\cdot{\bf b})}
	\tau_{g_\kappa(t)} \pi_{g_\kappa^\prime(t)} \mathcal{U}_{\kappa}(t) \pi_{-\bf a} \tau_{-\bf b} \varphi
\]
is valid for arbitrary $\varphi\in \Sigma^{1,1}$. 
Alternatively, let $\zeta=0$ in \eqref{eq:explicitsol}.
\end{proof}

\begin{proposition}\label{prop:phaseformula}
Let $w(t)=\mathcal{U}_\omega (t)\pi_{-\bf a} \tau_{-\bf b} u_0$.
Define $\psi_\omega(t)$ as in Theorem \ref{thm:harmonic}.
Then,
\[
	\int_0^t \int_{{\mathbb{R}}^d} |y|^2 |w(t)|^2 dy ds =
	\psi_\omega(t) -M \int_0^t |g_\omega(s)|^2 ds .
\]
\end{proposition}
\begin{proof}
By the previous proposition, we obtain
\[
	\Lebn{xw(t)}2^2 = \Lebn{x\tau_{-g_\omega(t)} \mathcal{U}_\omega(t)u_0}2^2
	=\Lebn{x\mathcal{U}_\omega(t)u_0}2^2 - M|g_\omega(t)|^2,
\]
where we have used $X[\mathcal{U}_\omega(t)u_0] = M g_\omega(t)$.
It therefore suffices to show that $\psi_\omega(t)=\int_0^t \Lebn{x\mathcal{U}_\omega(s)u_0}2^2ds$.

Assume $\omega>0$.
It is well known that $\mathcal{U}_{\omega}(t)$ is decomposed as
$\mathcal{U}_{\omega}(t)= \mathcal{M}_{\omega}(t)\mathcal{D}_{\omega}(t)
\mathcal{F} \mathcal{M}_{\omega}(t)$,
where $\mathcal{M}_{\omega}(t)$ is a multiplication operator defined by
\[
	\mathcal{M}_{\omega}(t)=  \exp\left(i\frac{\sqrt\omega}2 \coth(\sqrt\omega t) |x|^2\right)
\]
and $\mathcal{D}_{\omega}(t)$ is a dilation operator defined by
\[
	(\mathcal{D}_{\omega}(t)f)(x)=  \left( \frac{\sqrt\omega}{i \sinh (\sqrt\omega t)} \right)^{\frac{n}2}f\left(\frac{\sqrt\omega x}{\sinh(\sqrt\omega t)}\right).
\]
Hence, 
\begin{align*}
	\Lebn{x\mathcal{U}_\omega(t)u_0}2^2
	={}& \left(\frac{\sinh (\sqrt\omega t)}{\sqrt\omega}\right)^2 \Lebn{x\mathcal{F}\mathcal{M}_{\omega}(t) u_0}2^2 \\
	={}&\left(\frac{\sinh (\sqrt\omega t)}{\sqrt\omega}\right)^2 \Lebn{\nabla\left(\mathcal{M}_{\omega}(t) u_0\right)}2^2 \\
	={}& \left(\frac{\sinh (\sqrt\omega t)}{\sqrt\omega}\right)^2 \Lebn{\nabla u_0+i\sqrt\omega\coth(\sqrt\omega t)x u_0 }2^2 \\
	={}&  \frac{c}{\omega} \sinh^2(\sqrt\omega t)
	+ \frac{2 d}{\sqrt\omega} \sinh(\sqrt\omega t) \cosh(\sqrt\omega t) + e \cosh^2(\sqrt\omega t)
\end{align*}
and so
\begin{align*}
	\psi_\omega(t)
	= \left(\frac{c+e\omega}{2\omega^{3/2}}\right)\sinh(\sqrt\omega t) \cosh(\sqrt\omega t)
	+\frac{d}{\omega }\sinh^2(\sqrt\omega t) 
	-\left(\frac{c-e\omega }{2\omega }\right)t,
\end{align*}
where $c$, $d$, and $e$ are constants defined in \eqref{eq:cde}.
The proof for $\omega\leqslant 0$ is similar.
We omit details.
\end{proof}

\begin{proposition}\label{prop:motion}
Let $d\geqslant1$ and $\eta,\zeta\in{\mathbb{R}}$. Let $u_0\in \Sigma^{1,1}$.
Let $u\in C({\mathbb{R}}; \Sigma^{1,1})$ solve \eqref{eq:2H} in $(\Sigma^{1,1})^\prime$ sense.
Then, $P[u(t)]$ is a continuously differentiable function of time
and $\frac{d}{dt}P[u(t)]=\eta X[u(t)]$ holds.
\end{proposition}
\begin{proof}
Since $u\in C({\mathbb{R}};\Sigma^{1,1})$, we see that $X[u(t)] \in C({\mathbb{R}})$.
By \eqref{eq:2H}, $\Lebn{u(t)}2=\Lebn{u_0}2$
and $\frac{d}{dt}X[u(t)]=P[u(t)] \in C({\mathbb{R}})$ hold.
Set
\[
	v(t,x)= u(t,x) \exp\left( -i\frac{\zeta}{2}\int_0^t \int_{{\mathbb{R}}^d}|y|^2|u(t,s)|^2 dy\, ds \right).
\]
We have $|v(t,x)|=|u(t,x)|$ and so $X[v(t)]=X[u(t)]$
and $\Lebn{v(t)}2=\Lebn{u(t)}2=\Lebn{u_0}2$.
Hence, one sees that $v$ is a solution to
\[
	i\partial_t v + \frac12 \Delta v + \frac\omega2 |x|^2 v = \lambda X[v(t)] \cdot x v, \quad
	v(0)=u_0,
\]
where $\omega=\eta+\zeta M$.
Let us introduce a function $H(t)\in C^3({\mathbb{R}})$ as follows:
\[
	H(t)= -\zeta \int_0^t \frac{e^{\sqrt\omega (t-s)}+ e^{-\sqrt\omega (t-s)}}{2}
	\int_0^s X[u(\sigma)] d\sigma\, ds ,
\]
where $\frac{e^{\sqrt\omega t}+ e^{-\sqrt\omega t}}{2} = \cos (\sqrt{|\omega|}t)$
if $\omega<0$.
Remark that $H$ is a solution of 
$H^{\prime\prime}(t)=\omega H(t) - \zeta X[u(t)]$ with $H(0)=H^\prime(0)=0$.
Further set
\[
	v(t,x) = \exp\left(-\frac{i}{2}\int_0^t (|H^\prime(s)|^2 + \omega |H(s)|^2) ds\right)
	[\pi_{H^\prime(t)} \tau_{H(t)} w(t)](x).
\]
Then,
$w$ solves
$	i\partial_t w + \frac12 \Delta w + \frac{\omega}{2}|x|^2 w 
	=0$ with $w(0) = u_0$, and so
 $X[w(t)] = M g_\omega (t) \in C^\infty({\mathbb{R}})$ follows.
Hence, we conclude that
\[
	X[u(t)]=X[v(t)] = \int_{{\mathbb{R}}^d} (y+H(t)) |w(t,y)|^2 dy = M g_\omega(t) + M H(t) \in C^3({\mathbb{R}}),
\]
which gives us $P[u(t)] = \frac{d}{dt}X[u(t)]\in C^2({\mathbb{R}})$. Moreover,
\begin{align*}
	\frac{d}{dt}P[u(t)]=\frac{d^2}{dt^2} X[u(t)] = {}&\frac{d^2}{dt^2} X[w(t)] + M H^{\prime\prime}(t) \\
	={}& M \omega g_\omega(t) + M (\omega H(t) - \zeta X[u(t)]) \\
	={}& \omega (M g_\omega(t)+M H(t)) - \zeta M X[u(t)] \\
	={}& (\omega - \zeta M) X[u(t)] \\
	={}& \eta X[u(t)].
\end{align*}
\end{proof}

\section{\texorpdfstring{Proof of Theorem \ref{thm:main1}}{Proof of Theorem 1.1}}
We shall prove Theorem \ref{thm:main1} in a general framework.
Consider a \emph{generalized Hartree equation}
\begin{equation}\label{eq:gH}\tag{$\mathrm{gH}$}
\left\{
\begin{aligned}
&i \partial_t u + \frac{1}2 \Delta u = - ((V+R)* |u|^2)u , \text{ in } {\mathbb{R}}^{1+d},\\
& u(0,x)=u_0,
\end{aligned}
\right.
\end{equation}
where $V$ and $R$ are  real-valued functions of $x\in{\mathbb{R}}^d$.

A pair $(q,r)$ is admissible if $2\leqslant q,r\leqslant \infty$ satisfy the relation
$2/q=\delta(r):=d(1/2-1/r) \in [0,1]$
(however, we exclude the case $(d,q,r)=(2,2,\infty)$).
For nonnegative numbers $p$, $q$ and $r$, define
\begin{equation*}
	W^{p,q}_{V,r} := \{ f \in L^r({\mathbb{R}}^d); 
	\Jbr{\nabla}^{p}f, \, \Jbr{V(\cdot)}^{q/2} f \in L^r({\mathbb{R}}^d) \}
\end{equation*}
with a norm
\[
	\norm{f}_{{W}^{p,q}_{V,r}} := \Lebn{\Jbr{\nabla}^{p} f}r + \Lebn{\Jbr{V(\cdot)}^{q/2} f}r.
\]
Let $\Sigma^{p,q}_V:=W^{p,q}_{V,2}$.

Take a vector-valued function $W$ and set
\begin{equation}\label{def:K}
	K(x,y)=V(x-y)- V(x) + y\cdot W(x) .
\end{equation}
The assumptions on the potential $V$ and $R$ are the following.
\begin{assumption}
Suppose that $V:{\mathbb{R}}^d\to{\mathbb{R}}$ is a smooth function
satisfying the following properties:
\begin{itemize}
\item[$(\rm{V1})$] 
	$\partial^\alpha V(x) \in L^\infty({\mathbb{R}}^d)$
	for all index $\alpha$ with $|\alpha|\geqslant 2$;
\item[$(\rm{V2})$] There exist constants $C>0$ and $\kappa\in[0,1)$ such that 
 $\abs{\nabla V(x)} \leqslant C\Jbr{V(x)}^{\kappa/2}$;
\item[$(\rm{V3})$] There eixsts a vector-valued function $W$ such that 
$\partial^\alpha W(x) \in L^\infty({\mathbb{R}}^d)$ for all $|\alpha|\geqslant 2$ and
$K$ defined by \eqref{def:K} satisfies
\begin{equation*}
	\sup_x\abs{ K(x,y)} + \sup_x\abs{ \nabla_x K(x,y)} \leqslant C \Jbr{V(y)};
\end{equation*}
\item[$(\rm{V4})$] There exists a constant $C>0$ such that $\Jbr{x} \leqslant C \Jbr{V(x)}$.
\end{itemize}
\end{assumption}
\begin{assumption}
	Assume that $R$ satisfies the following.
\begin{itemize}
	\item[$(\rm{R1})$] $R \in L^\zeta({\mathbb{R}}^d) + L^\infty({\mathbb{R}}^d)$, where
	$\zeta \in [1,\infty]$ and $\zeta>d/4$.
	\item[$(\rm{R2})$] $R^+:=\max(R,0)\in L^\theta({\mathbb{R}}^d) + L^\infty({\mathbb{R}}^d)$, where
	$\theta \in [1,\infty]$ and $\theta>d/2$.
\end{itemize}
\end{assumption}
\begin{remark}
\begin{enumerate}
\item Roughly speaking, $V$ and $R$ denote a divergent part and a remainder
of a (divergent) potential $V+R$, respectively.
When we prove Theorem \ref{thm:main1}, we choose
$V=\lambda |x|^\gamma \chi(x)$ and $R=\lambda|x|^\gamma (1-\chi(x))$, where
$\chi$ is a smooth radial function such that $\chi\equiv1$ for $|x|\geqslant2$
and $\chi\equiv0$ for $|x|\leqslant1$.
Notice that $\lambda|x|^\gamma$ itself does not satisfy $(\rm{V1})$.
\item $\Sigma^{1,1}_V \subset \Sigma^{1,1/2}$ as long as Assumption $(\rm{V4})$ is satisfied,
under which condition $X[u]$ is well-defined for $u \in \Sigma^{1,1}_V$.
\item If $(\rm{R1})$ is satisfied then $(R*|u|^2)u \in (H^1)^\prime$ for $u\in H^1$.
An example of $R$ satisfying $(\rm{R1})$ is the $H^1$-subcritical
Hartree-type nonlinearity; $\eta |x|^{-\nu}$ with $\eta\in{\mathbb{R}}$ and $0<\nu<\min(4,d)$.
\end{enumerate}
\end{remark}

The main result of this section is the following.
\begin{theorem}\label{thm:general}
Suppose Assumptions $(\rm{V1})$, $(\rm{V2})$, $(\rm{V3})$, $(\rm{V4})$, $(\rm{R1})$, and $(\rm{R2})$ are satisfied.
Then, \eqref{eq:gH} is globally well-posed in $\Sigma^{1,1}_V$.
More precisely,
for  $u_0\in \Sigma^{1,1}_{V}$, there exists a global solution $u$ of \eqref{eq:gH}
in $C({\mathbb{R}}; \Sigma^{1,1}_{V})\cap C^1({\mathbb{R}};(\Sigma^{1,1}_V)^\prime)$ which
satisfies
	$u \in L^q_{\mathrm{loc}}({\mathbb{R}};W^{1,1}_{V,r})$
for all admissible pair $(q,r)$ and
conserves the mass $\Lebn{u(t)}2$, the \emph{energy}
\begin{equation}\label{def:genergy}
	E[u(t)] = \frac12 \Lebn{\nabla u(t)}2^2 - \frac14
	\iint_{{\mathbb{R}}^{d+d}} (V(x-y)+R(x-y)) |u(t,x)|^2 |u(t,y)|^2 dxdy,
\end{equation}
and the momentum $P[u(t)]$.
The solution is unique in the class
\[
	\left\{  u\in L^\infty({\mathbb{R}},\Sigma^{1,1}_V)\cap L^{\frac{8\zeta}d}_{\mathrm{loc}}({\mathbb{R}},W^{1,1}_{V,\frac{4\zeta}{2\zeta-1}}) \Bigm| P[u(t)] = \mathrm{const.} \right\}.
\]	
Moreover, we have the following estimates:
If $V \leqslant 0$ then
\[
	\Lebn{\nabla u(t)}2 \leqslant C, \quad
\Lebn{\Jbr{V(\cdot)}^{1/2} u(t)}2\leqslant C\Jbr{t}^{\frac1{2-\kappa}};
\]
otherwise
\[
	\norm{\nabla u(t)}_{L^2}
	\leqslant C \Jbr{t}^{\frac{1}{1-\kappa}}, \quad
	\Lebn{\Jbr{V(\cdot)}^{1/2} u(t)}2\leqslant C\Jbr{t}^{\frac1{1-\kappa}},
\]
where $\kappa$ is the number defined in Assumption $(\rm{V2})$.
\end{theorem}

We now apply the transform observed in the introduction.
Then, the problem boils down to the Cauchy problem of
the following modified version of the \eqref{eq:gH}:
\begin{equation}\label{eq:gmH}\tag{$\mathrm{mgH}$}
\left\{
\begin{aligned}
&i \partial_t u + \frac{1}2 \Delta u + M V(x) u= - u \int_{{\mathbb{R}}^d} K(\cdot,y) |u(y)|^2 dy  -(R*|u|^2)u,\\
& u(0)=u_0 \in \Sigma^{1,1}_V,
\end{aligned}
\right.
\end{equation}
where $M$ is as in \eqref{def:M} and
$K$ is defined by \eqref{def:K} with $W$ given in Assumption $(\rm{V3})$.
If $(\rm{V3})$ is satisfied then $u\int K(\cdot,y)|u(y)|^2 dy \in \Sigma^{j,1}_V$ for
$u\in \Sigma^{j,1}_V$, where $j=0,1$.

Denote $A:=\frac12\Delta + M V(x)$.
It is well known that $A$ is an essentially self-adjoint operator on $L^2$ as long as
Assumption $(\rm{V1})$ holds.
Moreover, we have Strichartz's estimate under this condition (\cite{YajCMP}).
\begin{lemma}[Strichartz's estimate]
Suppose $(\rm{V1})$. For any fixed $T>0$, the following properties hold:
\begin{itemize}
\item 
Suppose $\varphi \in L^2({\mathbb{R}}^2)$.
For any admissible pair $(q,r)$, there exists a constant $C=C(T,q,r)$
such that
\[
	\norm{e^{itA} \varphi}_{L^q((-T,T);L^r)} \leqslant C \Lebn{\varphi}2.
\]
\item Let $I \subset (-T,T)$ be an interval and $t_0 \in \overline{I}$. 
For any admissible pairs $(q,r)$ and $(\gamma,\rho)$,
there exists a constant $C=C(t,q,r,\gamma,\rho)$ such that
\[
	\norm{\int_{t_0}^t e^{i(t-s)A} F(s) ds}_{L^q(I;L^r)}
	\leqslant C \norm{F}_{L^{\gamma^\prime}(I; L^{\rho^\prime})}
\]
for every $F \in L^{\gamma^\prime}(I; L^{\rho^\prime})$.
\end{itemize}
\end{lemma}
\begin{lemma}\label{lem:commutatorA}
Let $A=\frac12\Delta + M V(x)$.
Suppose that Assumption $(\rm{V1})$ holds.
Denote by $e^{itA}$ a one-parameter group generated by $A$.
Let $F$ be an arbitrary weight function such that $\nabla F$ and $\Delta F$ are bounded.
Then, for all $f \in \Sigma^{1,1}_V$,
\[
	e^{- itA} \nabla e^{itA} f= \nabla f + iM\int_0^t e^{-isA} \left(\nabla V\right) e^{isA} f ds
\]
and
\[
	e^{-itA}F e^{itA} f=  F f + i\int_0^t e^{-isA} 
	\left(\nabla F\cdot \nabla + \frac12 (\Delta F)\right)e^{isA} f ds.
\]
\end{lemma}
By means of this lemma, it immediately follows from $(\rm{V1})$ and $(\rm{V2})$
that
\[
	\norm{e^{itA}\phi}_{L^\infty((-T,T),\Sigma^{1,1}_V)}
	\leqslant \norm{\phi}_{\Sigma^{1,1}_V} + CT \norm{e^{itA}\phi}_{L^\infty((-T,T),\Sigma^{1,1}_V)} + CT\Lebn{\phi}2.
\]
This implies $\norm{e^{itA}\phi}_{\Sigma^{1,1}_{V}} \leqslant C(|t|,M)\norm{\phi}_{\Sigma^{1,1}_{V}}$.
A similar argument shows
$\norm{e^{itA}\phi}_{\Sigma^{k,k}_{V}} \leqslant C(k,|t|,M)\norm{\phi}_{\Sigma^{k,k}_{V}}$
for any positive integer $k$.
Further, as in \cite[Lemma 2.11]{DR-JMP}, $e^{itA}\phi$
is continuous in $\Sigma^{k,k}_V$
with respect to $M$ for each $t$ and $\phi\in \Sigma^{k,k}_V$.

\subsection{\texorpdfstring{Local well-posedness of \eqref{eq:gmH}}{Local well-posedness of (mgH)}}
We first give a unique local solution to \eqref{eq:gmH}.
Throughout this subsection and the next subsection
we suppose that the constant $M$ in the operator $A$ is not
necessarily equal to $\Lebn{u_0}2^2$.
In what follows, we write $L^p((-T,T);X)=L^p_T X$, for short.
The integral form of \eqref{eq:gmH} is
\begin{multline}\label{def:Q}
	u(t) = Q[u]:= e^{itA}u_0 +i \int_0^t e^{i(t-s)A} u(s)\int_{{\mathbb{R}}^d} K(x,y)|u(s,y)|^2dy ds\\
	+i \int_0^t e^{i(t-s)A} ((R*|u|^2)u)(s) ds.
\end{multline}

\begin{proposition}
Suppose that Assumptions $(\rm{V1})$, $(\rm{V2})$, $(\rm{V3})$, and $(\rm{R1})$ are satisfied.
Let $u_0\in \Sigma^{1,1}_V$.
Define a Banach space
\[
	\mathcal{H}_{T,\delta} := \{f \in L^\infty((-T,T);\Sigma^{1,1}_V);\,
	\norm{f}_{\mathcal{H}_T} \leqslant \delta \},
\]
where
\[
	\norm{f}_{\mathcal{H}_T} :=
	\norm{f}_{L^\infty_T {\Sigma}^{1,1}_V}
	+ \norm{f}_{L^q_T W^{1,1}_{V,r}}
\]
with $q=8\zeta/d$ and $r=4\zeta/(2\zeta-1)$.
Then, 
there exists $\delta_0$ depending only on $\norm{u_0}_{\Sigma^{1,1}_V}$
such that for any $\delta \in [\delta_0, \infty)$,
the operator $Q$ given in \eqref{def:Q} is a contraction map
from $\mathcal{H}_{T,\delta}$ to itself for suitable $T=T(\delta)>0$.
\end{proposition}

\begin{proof}
We prove that $Q$ is a contraction map in two steps.

\subsubsection*{Step 1}
Fix $\delta>0$.
We show that the existence of $T$ such that
$Q[u] \in \mathcal{H}_{T,\delta}$ as long as $u \in \mathcal{H}_{T,\delta}$.

Let $u \in  \mathcal{H}_{T,\delta}$.
We first establish estimates on $R$.
Set $R=R_1 + R_2$ with $R_1 \in L^\zeta$ and $R_2\in L^\infty$.
One sees from Young's inequality, H\"odler's inequality, and Sobolev's embedding that
\begin{align*}
	\norm{\mathcal{A}((R_1*|u|^2)u)}_{L^{q^\prime}_TL^{r^\prime}}
	&{}\leqslant C T^{1-\frac{d}{4\zeta}}\Lebn{R_1}\zeta \norm{\nabla u}_{L^\infty_T L^2}^{\frac{d}{2\zeta}}
	\norm{u}_{L^\infty_T L^2}^{2-\frac{d}{2\zeta}} \norm{\mathcal{A}u}_{L^q_TL^r}\\
	&{}\leqslant CT^{1-\frac{d}{4\zeta}} \delta^3
\end{align*}
and
\begin{align*}
	\norm{\mathcal{A}((R_2*|u|^2)u)}_{L^1_TL^2}
	&{}\leqslant C T \Lebn{R_2}\infty \norm{u}_{L^\infty_T L^2}^2
	\norm{\mathcal{A}u}_{L^\infty_T L^2}\\
	&{}\leqslant CT \delta^3,
\end{align*}
where $\mathcal{A}=\mathrm{Id}$, $\nabla$ or $\Jbr{V(x)}^{1/2}$.
By Strichartz's estimate, we have
\begin{multline*}
	\norm{Q[u]}_{L^\infty_T L^2} + \norm{Q[u]}_{L^q_T L^r} \leqslant 
	C\Lebn{u_0}2 + C\norm{u\int_{{\mathbb{R}}^d} K(\cdot,y)|u(y)|^2dy}_{L^1_TL^2}\\
	+C\norm{(R_1*|u|^2)u}_{L^{q^\prime}_TL^{r^\prime}}+
	C\norm{(R_2*|u|^2)u}_{L^1_TL^2}.
\end{multline*}
Since $\sup_x |K(x,y)| \leqslant C\Jbr{V(y)}$ holds
by Assumption $(\rm{V3})$, we have
\begin{align*}
	\Lebn{u\int_{{\mathbb{R}}^d} K(\cdot,y)|u(y)|^2dy}2 
	\leqslant{}&\norm{u(x)K(x,y)|u(y)|^2}_{L^2_x({\mathbb{R}}^d; L^1_y({\mathbb{R}}^d))}\\
	\leqslant{}&\norm{u(x)K(x,y)|u(y)|^2}_{L^1_y({\mathbb{R}}^d; L^2_x({\mathbb{R}}^d))}\\
	\leqslant{}& C \norm{u}_{L^2} \norm{\Jbr{V(\cdot)}^{1/2}u}_{L^2}^2.
\end{align*}
Take $L^1_T$ norm to yield
\[
	\norm{u\int_{{\mathbb{R}}^d} K(\cdot,y)|u(y)|^2dy}_{L^1_TL^2} \leqslant C
	T\norm{u}_{L^\infty_T L^2} \norm{u}_{L^\infty_T \Sigma^{1,1}_V}^2.
\]
Hence
\begin{equation}\label{eq:estQ1}
	\norm{Q[u]}_{L^\infty_T L^2}+\norm{Q[u]}_{L^q_T L^r} 
	\leqslant C\Lebn{u_0}2+ 
	C(T+T^{1-\frac{d}{4\zeta}})\delta^3.
\end{equation}
\smallbreak

We next estimate $\nabla Q[u]$.
By Lemma \ref{lem:commutatorA}, one sees that
\begin{align*}
	\nabla Q [u] ={}& e^{itA} \nabla u_0 + i \int_0^t e^{i(t-s)A} \nabla \left(u(s)\int_{{\mathbb{R}}^d} K(x,y)|u(s,y)|^2dy\right) ds \\
	&{} + i \int_0^t e^{i(t-s)A} \nabla \left( (R*|u|^2)u \right)ds \\
	&{}
	+ i M\int_0^t e^{i(t-s)A} (\nabla V) Q[u] (s)ds.
\end{align*}
By Strichartz's estimate,
\begin{align*}
	&\norm{\nabla Q [u]}_{L^\infty_TL^2} + \norm{\nabla Q [u]}_{L^q_TL^r}\\
	&{}\leqslant C\norm{\nabla u_0}_{L^2} + C\norm{\nabla \left(u(s)\int_{{\mathbb{R}}^d} K(x,y)|u(s,y)|^2dy\right)}_{L^1_TL^2}
	\\&\quad{}
	+C\norm{\nabla ((R_1*|u|^2)u)}_{L^{q^\prime}_TL^{r^\prime}}
	+C\norm{\nabla ((R_2*|u|^2)u)}_{L^1_TL^2}
	\\&\quad{}
	+CM \norm{(\nabla V)Q[u]}_{L^1_T L^2}.
\end{align*}
Using Assumption $(\rm{V3})$, we obtain
\begin{align*}
	&\norm{\left( \int_{{\mathbb{R}}^d} (\nabla K(\cdot,y))|u(y)|^2dy\right)u}_{L^1_TL^2}
	\leqslant CT\norm{\Jbr{V(\cdot)}^{1/2}u}_{L^\infty_TL^2}^2 \norm{u}_{L^\infty_TL^2},\\
	&\norm{(\nabla u)\int_{{\mathbb{R}}^d} K(\cdot,y)|u(y)|^2dy}_{L^1_TL^2}
	\leqslant CT\norm{\Jbr{V(\cdot)}^{1/2}u}_{L^\infty_TL^2}^2 \norm{\nabla u}_{L^\infty_TL^2}.
\end{align*}
It follows from Assumption $(\rm{V2})$ that
\[
	\norm{(\nabla V)Q[u]}_{L^1_T L^2}
	\leqslant C\norm{\Jbr{V(\cdot)}^{1/2}Q[u]}_{L^1_T L^2}.
\]
We hence deduce that
\begin{multline}\label{eq:estQ2}
	\norm{\nabla Q[u]}_{L^\infty_TL^2} + 
	\norm{\nabla Q[u]}_{L^q_TL^r} \\
	\leqslant C\Lebn{\nabla u_0}2
	+ C(T+T^{1-\frac{d}{4\zeta}})\delta^3 + C M T \norm{\Jbr{V(\cdot)}^{1/2}Q[u]}_{L^\infty_TL^2}.
\end{multline}
\smallbreak

Let us proceed to the estimate of $\Jbr{V(x)}^{1/2}Q[u]$.
Apply the second identity of Lemma \ref{lem:commutatorA} with
$F=\Jbr{V(x)}^{1/2}$ to yield
\begin{align*}
	&\Jbr{V(x)}^{1/2}Q[u] \\
	&{}= e^{itA} \Jbr{V(x)}^{1/2} u_0 \\
	&\quad{} + i\int_0^t e^{i(t-s)A} \Jbr{V(x)}^{1/2} u(s)\int_{{\mathbb{R}}^d} K(x,y)|u(s,y)|^2dy  ds\\
	&\quad{}+ i\int_0^t e^{i(t-s)A} \Jbr{V(x)}^{1/2} (R*|u|^2)u ds\\
	&\quad{}+i \int_0^t e^{i(t-s)A}\left(\nabla \Jbr{V(x)}^{1/2} \cdot \nabla + \frac12(\Delta \Jbr{V(x)}^{1/2})\right)Q[u](s)ds.
\end{align*}
Now, Assumptions $(\rm{V1})$ and $(\rm{V3})$ give us that
\[
	\abs{\nabla \Jbr{V(x)}^\frac12} \leqslant C, \quad
	\abs{\Delta \Jbr{V(x)}^\frac12} \leqslant C.
\]
Hence,
\begin{multline}\label{eq:estQ3}
	\norm{\Jbr{V(\cdot)}^{1/2}Q[u]}_{L^\infty_T L^2}
	+ \norm{\Jbr{V(\cdot)}^{1/2}Q[u]}_{L^q_T L^r}\\
	\leqslant C\Lebn{\Jbr{V(\cdot)}^{1/2} u_0}2
	+ C(T+T^{1-\frac{d}{4\zeta}})\delta^3 \\
	+CT(\norm{\nabla Q[u]}_{L^\infty_TL^2}+\norm{Q[u]}_{L^\infty_TL^2}). 
\end{multline}
From \eqref{eq:estQ1}, \eqref{eq:estQ2}, and \eqref{eq:estQ3},
we finally reach to 
\[
	\norm{Q[u]}_{\mathcal{H}_T}
	\leqslant C_1 \norm{u_0}_{\Sigma^{1,1}_V}
	+ C_2 (T+T^{1-\frac{d}{4\zeta}})\delta^3
	+ C_3 T \norm{Q[u]}_{\mathcal{H}_T}.
\]
Letting $T$ so small that $C_3 T \leqslant 1/2$, we see
$\norm{Q[u]}_{\mathcal{H}_T}
	\leqslant 2C_1 \norm{u_0}_{\Sigma^{1,1}_V}
	+ 2C_2 (T+T^{1-\frac{d}{4\zeta}})\delta^3$.
We now choose $\delta_0:= 3C_1 \norm{u_0}_{\Sigma^{1,1}_V}$.
Then, for any $\delta\geqslant \delta_0$,
$Q$ maps $\mathcal{H}_{T,\delta}$ to itself, provided $T$
is so small that $2C_2(T+T^{1-\frac{d}{4\zeta}})\delta^2
\leqslant 1/3$.

\subsubsection*{Step 2}
We next show that $Q$ is a contraction map.
In the same way as in Step 1, we obtain
\begin{equation}\label{eq:contraction}
	\norm{Q[u] -Q[v]}_{\mathcal{H}_{T}}
	\leqslant C_4 (T+T^{1-\frac{d}{4\zeta}}) (\norm{u}_{\mathcal{H}_{T}}^2+\norm{v}_{\mathcal{H}_{T}}^2)\norm{u-v}_{\mathcal{H}_{T}}
\end{equation}
for small $T$ and $u,v \in \mathcal{H}_{T,\delta}$.
Letting $T$ so small that $2C_4(T+T^{1-\frac{d}{4\zeta}})\delta^2\leqslant 1/2$ if necessary,
we conclude that
$\norm{Q[u] -Q[v]}_{\mathcal{H}_{T}}
	\leqslant (1/2) \norm{u-v}_{\mathcal{H}_{T}}$.
\end{proof}

\begin{proposition}
Suppose that Assumptions $(\rm{V1})$, $(\rm{V2})$, $(\rm{V3})$, $(\rm{V4})$, and $(\rm{R1})$ are satified.
Let $u_0\in \Sigma^{2,2}_V$.
Define a Banach space
\[
	\mathcal{I}_{T,\delta} := \{f \in L^\infty((-T,T);\Sigma^{2,2}_V);\,
	\norm{f}_{\mathcal{I}_T}
	\leqslant \delta \},
\]
where
	$\norm{f}_{\mathcal{I}_T} :=
	\norm{f}_{L^\infty_T \Sigma^{2,2}_V}
	+ \norm{f}_{L^q_T W^{2,2}_{V,r}}
$
with $q=8\zeta/d$ and $r=4\zeta/(2\zeta-1)$.
Then, there exists $\delta_0$ depending only on $u_0$
such that for any $\delta \in [\delta_0, \infty)$,
the operator $Q$ given in \eqref{def:Q} is a contraction map
from $\mathcal{I}_{T,\delta}$ to itself for suitable $T=T(\delta)$.
\end{proposition}
The proof is done in a similar way. 

\begin{theorem}\label{thm:LWPgmH}(1)
Suppose that Assumptions $(\rm{V1})$, $(\rm{V2})$, $(\rm{V3})$, and $(\rm{R1})$ are satisfied.
Then, for any $u_0 \in \Sigma^{1,1}_V$, there exists
$T=T(\norm{u_0}_{\Sigma^{1,1}_V})$ and a unique solution
$u \in C([-T,T],\Sigma^{1,1}_V)
\cap L^{8\zeta/d}([-T,T],W^{1,1,}_{V,\frac{4\zeta}{2\zeta-1}})$ of \eqref{eq:gmH}.
The solution belongs to
 $ C^1([-T,T],(\Sigma^{1,1}_V)^\prime) $ and
$ L^q([-T,T];W^{1,1}_{V,r})$ for all admissible pair $(q,r)$,
and conserves mass $\Lebn{u(t)}2$.
Moreover, the solution depends continuously on the data.

(2) If Assumptions $(\rm{V1})$, $(\rm{V2})$, $(\rm{V3})$, $(\rm{V4})$, and $(\rm{R1})$ are satisfied and if
\begin{equation}\label{eq:neutrality}
	\Lebn{u_0}2^2=M,\quad X[u_0] = P[u_0] =0
\end{equation} 
are fulfilled, then the solution given in (1) 
conserves energy $E[u(t)]$ defined in \eqref{def:genergy} and momentum $P[u(t)]$.
In other words, the Cauchy problem \eqref{eq:gmH} is locally well-posed in 
$\widetilde\Sigma^{1,1}_{V}:=\{u_0\in \Sigma^{1,1}_{V}| u_0 \text{ satisfies }\eqref{eq:neutrality}\}$.
Moreover, the solution $u$ solves \eqref{eq:gH}.
\end{theorem}
\begin{remark}
The theorem holds even if we replace Assumption $(\rm{V2})$ with a weaker one;
(V2') $|\nabla V(x)| \leqslant C \Jbr{V(x)}^{1/2}$.
\end{remark}
\begin{proof}
For $u_0 \in \Sigma^{1,1}_V$, we choose $\delta$ and $T$ so that
$Q[u]$ becomes a contraction map from $\mathcal{H}_{T,\delta}$ to itself.
Then, we obtain a unique solution $u \in C([-T,T],\Sigma^{1,1}_V)\cap L^{8\zeta/d}([-T,T],W^{1,1,}_{V,\frac{4\zeta}{2\zeta-1}})$
of integral version of \eqref{eq:gmH}.
From the equation \eqref{eq:gmH}, we see that
$u \in C^1((-T,T),(\Sigma^{1,1}_V)^\prime)$.
Strichartz's estimate gives a bound in $L^q([-T,T],W^{1,1}_{V,r})$ for all admissible pair.
As in \eqref{eq:contraction}, one gets
\begin{equation}\label{eq:approximation}
	\norm{u -v}_{\mathcal{H}_{T}}
	\leqslant C \norm{u(0) -v(0)}_{\Sigma^{1,1}_V} +
C(T+T^{1-\frac{d}{4\zeta}}) (\norm{u}_{\mathcal{H}_{T}}^2+\norm{v}_{\mathcal{H}_{T}}^2)\norm{u-v}_{\mathcal{H}_{T}}
\end{equation}
for any two solutions $u$ and $v$ of \eqref{eq:gmH}.
From this estimate, we deduce continuous dependence of the solution on the data.
Now, a use of \eqref{eq:gmH} shows that
\[
	\frac{d}{dt}\Lebn{u(t)}2^2 = 2\operatorname{Re} \Jbr{u,\partial_t u}_{\Sigma^{1,1}_V,( \Sigma^{1,1}_V)^\prime}=0,
\]
which completes the proof of the first statement of the theorem.

We now suppose that Assumption $(\rm{V4})$ holds and $u_0$ satisfies \eqref{eq:neutrality}.
To justify the momentum conservation, we use a regularization argument.
Let $\{u_{0,n}\}_{n}$ be a sequence of functions in $\Sigma^{2,2}_V$ which 
satisfies $\Lebn{u_{0,n}}2^2=M$ and
converges to $u_0$ in $\Sigma^{1,1}_V$ as $n\to\infty$,
and let $u_n \in C([-\widetilde{T}_n,\widetilde{T}_n],\Sigma^{2,2}_V)
\cap C^1((-\widetilde{T}_n,\widetilde{T}_n),L^2)$ be
corresponding solutions of \eqref{eq:gmH} with $u_n(0)=u_{0,n}$.
Notice that $\widetilde{T}_n \geqslant T/2$ for large $n$ since 
$\norm{u_{0,n}}_{\Sigma^{1,1}_V}$ converges to
$\norm{u_{0}}_{\Sigma^{1,1}_V}$ as $n\to\infty$.
Thanks to \eqref{eq:approximation}, $u_n$ converges to $u$ in
\begin{equation}\label{eq:app-convergence}
	C([-T/2,T/2],\Sigma^{1,1}_V) \cap C^1((-T/2,T/2),(\Sigma^{1,1}_V)^\prime).
\end{equation}
It therefore holds for each $n$ that
\begin{align*}
	\frac{d}{dt}
	P[u_n(t)]
	={}& \operatorname{Im} \int_{{\mathbb{R}}^d} \overline{\partial_t u_n(t,y)}
	\nabla u_n(t,y) dy\\
	{}&+ \operatorname{Im} \int_{{\mathbb{R}}^d} \overline{u_n(t,y)}
	\nabla \partial_t u_n(t,y) dy,
\end{align*}
which makes sense because $\nabla \partial_t u_n \in H^{-1}$.
Integrating by parts and plugging \eqref{eq:gmH}, we see that
\begin{align}\label{eq:Pconservation1}
	\frac{d}{dt}P[u_n(t)] = \int_{{\mathbb{R}}^d} |u_n(t,x)|^2 \nabla \left(W(x)
	\cdot X[u_n(t)] \right)dx.
\end{align}
Notice that the right hand side makes sense because $|\nabla W_j(x)| \leqslant C\Jbr{V(x)}^{\frac12}$ is valid
and $\Sigma^{1,1}_V\subset \Sigma^{1,1/2}$ holds by virtue of $(\rm{V4})$,
where $W_j$ denotes the $j$-th component of $W$.
Indeed, take $y={\bf e}_j$ in \eqref{def:K} 
and differentiate with respect to $x$ to yield
$\nabla W_j(x)=\nabla K(x,{\bf e}_j)-\nabla V(x-{\bf e}_j)+\nabla V(x)$.
By Assumptions $(\rm{V1})$, $(\rm{V2})$, and $(\rm{V3})$, one obtains $|\nabla W_j(x)| \leqslant C\Jbr{V(x)}^{1/2}$.
We also have 
\begin{equation}\label{eq:Pconservation2}
	\frac{d}{dt}X[u_n(t)]
	= P[u_n(t)].
\end{equation}
Two estimates \eqref{eq:Pconservation1} and \eqref{eq:Pconservation2} give us
\begin{multline*}
	\frac{d}{dt} \left( \abs{X[u_n(t)]}^2 + |P[u_n(t)]|^2 \right)\\
	\leqslant C\left(1+\norm{u_n(t)}_{\Sigma^{0,1}_{V}}^2\right) \left( \abs{X[u_n(t)]}^2 + |P[u_n(t)]|^2 \right).
\end{multline*}
Applying Gronwall's lemma, we conclude that 
\begin{equation}\label{eq:pseudomomentum}
	\abs{X[u_n(t)]}^2 + |P[u_n(t)]|^2
	\leqslant C \delta^2 e^{T(1+\delta^2)} \norm{u_n - u_0}_{\Sigma^{1,1}_V}^2
\end{equation}
for $t\in(-T,T)$,
where we have used the estimates
\begin{align*}
	\abs{X[u_{n}(0)]}
	&{}=\abs{X[u_{0,n}]-X[u_0]}=\abs{\int_{{\mathbb{R}}^d} y (|u_{0,n}(y)|^2 - |u_{0}(y)|^2)dy} \\
	&{}\leqslant C\norm{u_{0,n}-u_0}_{\Sigma^{1,1}_V}(\norm{u_{0,n}}_{\Sigma^{1,1}_V}
	+\norm{u_0}_{\Sigma^{1,1}_V})
\end{align*}
and
\begin{align*}
	\abs{P[u_n(0)]}
	&{}=\abs{P[u_{0,n}]-P[u_0]}=\abs{\int_{{\mathbb{R}}^d} \xi (|{\mathcal{F}} u_{0,n}(\xi)|^2 - |{\mathcal{F}} u_{0}(\xi)|^2)d\xi} \\
	&{}\leqslant \norm{u_{0,n}-u_0}_{\Sigma^{1,1}_V}(\norm{u_{0,n}}_{\Sigma^{1,1}_V}
	+\norm{u_0}_{\Sigma^{1,1}_V}).
\end{align*}
Convergence of $u_n$ in the topology \eqref{eq:app-convergence} allows us to claim
$X[u(t)] \equiv P[u(t)]\equiv0$ for $t\in (-T/2,T/2)$.
By this fact and the mass conservation, the right hand side of \eqref{eq:gmH}
becomes
\begin{multline*}
	-u \left( (V*|u|^2)- V(x) \Lebn{u_0}2^2
	+W (x) \cdot  X[u]\right) \\
	= - (V*|u|^2)u +M V(x) u-0.
\end{multline*}
Therefore, $u$ solves \eqref{eq:gH}.
Let us again consider the above sequence $\{u_{0,n}\}_n \in \Sigma^{2,2}_V$ approximating $u_0$
and the corresponding sequence of solutions $\{u_n\}_n$ to \eqref{eq:gmH}.
Conservation of mass allows us to rewrite \eqref{eq:gmH} into
\[
	i\partial_t u_n + \frac12 \Delta u_n = -((V+R)*|u_n|^2) u_n - W \cdot X[u_n] u_n.
\]
Multiplying this equality by $\overline{\partial_t u_n}$ and integrating real parts
of the resulting terms, we get
\[
	\frac{d}{dt}E[u_n(t)]= \operatorname{Re} \int W u_n\overline{\partial_t u_n} dx \cdot X[u_n].
\]
This calculation makes sense because $\overline{\partial_t u_n}$
is a continuous $L^2$-valued function and $u_n$ satisfies
\eqref{eq:gmH} in $L^2$ sense.
Since $\operatorname{Re} \int W_j u_n\overline{\partial_t u_n} dx=\frac12\operatorname{Im}\int \overline{u_n}\nabla W_j \cdot \nabla u_ndx$,
one sees from \eqref{eq:pseudomomentum} that
\[
	\abs{E[u_n(t)] - E[u_{0,n}]}
	\leqslant \frac{C\delta^3 }{1+\delta^2}e^{\frac{T}2(1+\delta^2)}\norm{u_n - u_0}_{\Sigma^{1,1}_V}
	\to 0
\]
as $n\to\infty$ for $t\in (-T,T)$.
By means of the convergence of $u_n$ in \eqref{eq:app-convergence}, 
$|E[u(t)]-E[u(0)]|\leqslant |E[u(t)]-E[u_{n}(t)]|+ |E[u_n(t)]-E[u_{0,n}]| + |E[u_{0,n}]-E[u_{0}]|\to 0$ as $n\to\infty$ for $t\in (-T/2,T/2)$,
which gives us the desired conservation law.
\end{proof}

\subsection{\texorpdfstring{Global well-posedness of \eqref{eq:gmH}}{Global well-posedness of (mgH)}}
We next extend the above solution of \eqref{eq:gH} to whole real line ${\mathbb{R}}$.
\begin{lemma}\label{lem:estVu}
Suppose Assumption $(\rm{V2})$ is satisfied.
Let $u\in C((-T,T);\Sigma^{1,1}_V)$ be a solution to \eqref{eq:gH} with 
data $u_0 \in \Sigma^{1,1}_V$.
Then, it holds that
\[
	 \norm{\Jbr{V(\cdot)}^{1/2}u}_{L^\infty_t L^2}
	\leqslant C (\norm{u_0}_{\Sigma^{1,1}_V}) (1+ (|t|\norm{\nabla u}_{L^\infty_tL^2}
	)^{\frac1{2-\kappa}} )
\]
for $t\in(-T,T)$, where $\kappa$ is the number defined in Assumption $(\rm{V2})$.
\end{lemma}
\begin{proof}
Let us estimate $\Lebn{\sqrt{|V|}u}2$ because $\Lebn{u}2$ is conserved.
For $r>0$, we have
\begin{align*}
	\frac{d}{dt} \Lebn{{\bf 1}_{\{|x|\leqslant r\}}\sqrt{|V|} u(t)}2^2
	&{}= 2\operatorname{Re} \Jbr{{\bf 1}_{\{|x|\leqslant r\}} |V| u, \partial_t u  }
	= \operatorname{Im} \Jbr{{\bf 1}_{\{|x|\leqslant r\}}|V| u, \Delta u }\\
	&{}= \operatorname{Im} \int_{|x|\leqslant r} (\overline{\nabla u} \cdot \nabla |V|) u dx,
\end{align*}
which implies
\[
	\abs{\frac{d}{dt} \Lebn{{\bf 1}_{\{|x|\leqslant r\}}\sqrt{|V|} u(t)}2^2}
	\leqslant \Lebn{{\bf 1}_{\{|x|\leqslant r\}}|\nabla V| u(t)}2 \Lebn{\nabla u(t)}2
\]
Integrating in time and substituting $|\nabla V|\leqslant C\Jbr{V}^{\kappa/2}$ give us
\begin{multline*}
	\norm{{\bf 1}_{\{|x|\leqslant r\}}\sqrt{|V|} u}_{L^\infty_tL^2}^2 \leqslant
	C\norm{u_0}_{\Sigma^{1,1}_V}^2 \\
	+ |t|\norm{\nabla u}_{L^\infty_tL^2} \norm{{\bf 1}_{\{|x|\leqslant r\}}\Jbr{V}^{1/2} u}_{L^\infty_tL^2}^\kappa \norm{{\bf 1}_{\{|x|\leqslant r\}} u}_{L^\infty_tL^2}^{1-\kappa}.
\end{multline*}
Since $r$ is arbitrary, we pass to the limit $r\to \infty$ and reach to
\begin{equation}\label{eq:estVuind}
	\norm{\Jbr{V}^{1/2} u}_{L^\infty_tL^2}^2 \leqslant
	C\norm{u_0}_{\Sigma^{1,1}_V}^2
	+ C|t|\norm{\nabla u}_{L^\infty_tL^2}\Lebn{u_0}2^{1-\kappa}
	\norm{\Jbr{V}^{1/2} u}_{L^\infty_tL^2}^\kappa.
\end{equation}
It follows from Young's inequality that
\begin{multline*}
	C|t|\norm{\nabla u}_{L^\infty_tL^2} \Lebn{u_0}2^{1-\kappa}
	\norm{\Jbr{V}^{1/2} u}_{L^\infty_tL^2}^\kappa\\
	\leqslant \frac12 \norm{\Jbr{V}^{1/2} u}_{L^\infty_tL^2}^2
	+ C_{\kappa}(|t|\norm{\nabla u}_{L^\infty_tL^2}\Lebn{u_0}2^{1-\kappa})^{\frac2{2-\kappa}}.
\end{multline*}
Plugging this estimate to \eqref{eq:estVuind}, we conclude that
\[
	\norm{\Jbr{V}^{1/2} u}_{L^\infty_tL^2}^2
	\leqslant C\norm{u_0}_{\Sigma^{1,1}_V}^2
	+ C\Lebn{u_0}2^{2-\frac2{2-\kappa}}(|t|\norm{\nabla u}_{L^\infty_tL^2})^{\frac2{2-\kappa}}.
\]
\end{proof}

A blow-up criterion immediately follows from this lemma.
\begin{corollary}[Blow-up criterion]
Suppose Assumptions $(\rm{V1})$, $(\rm{V2})$, $(\rm{V3})$,  $(\rm{V4})$,
and $(\rm{R1})$ and \eqref{eq:neutrality} are satisfied.
Let $u\in C((-T_{\mathrm{min}},T_{\mathrm{max}});\Sigma^{1,1}_V)$ be a 
maximal solution to \eqref{eq:gH}.
If $T_{\mathrm{max}}<\infty$ (resp. $T_{\min}<\infty$) then
$\Lebn{\nabla u(t)}2\to\infty$ as $t \uparrow T_{\max}$
(resp. as $t \downarrow -T_{\min}$).
\end{corollary}
Now, the following lemma shows that \eqref{eq:gmH} is globally well-posed
in $\widetilde\Sigma^{1,1}_{V}$.
More precisely, the solution of \eqref{eq:gmH} given
in Theorem \ref{thm:LWPgmH} (2) never blows up in finite time.
\begin{lemma}[Global $\dot{H}^1$-bound]\label{lem:H1bound}
Suppose $(\rm{V1})$, $(\rm{V2})$, $(\rm{V3})$, $(\rm{V4})$, $(\rm{R1})$, and $(\rm{R2})$.
Let $u_0\in \widetilde\Sigma^{1,1}_{V}$ and let
$u\in C((-T_{\mathrm{min}},T_{\mathrm{max}});\Sigma^{1,1}_V)$ be a corresponding
maximal solution to \eqref{eq:gmH} which conserves energy $E[u(t)]$
and momentum $P[u(t)]$.
Then, we have the following bounds for $t\in(-T_{\mathrm{min}},T_{\mathrm{max}})$:
\begin{itemize}
\item If $V \leqslant 0$ then, $\Lebn{\nabla u(t)}2 \leqslant C$;
\item otherwise, $\norm{\nabla u(t)}_{L^2}
	\leqslant C \Jbr{t}^{\frac1{1-\kappa}}$,
where $\kappa$ is the number defined in Assumption $(\rm{V2})$.
\end{itemize}
\end{lemma}
\begin{proof}
If $V\leqslant 0$ then 
\begin{align*}
	\Lebn{\nabla u(t)}2^2 \leqslant{}& 2E[u(t)] + \frac12 \Lebn{ (R^+* |u|^2) |u|^2}1\\
	\leqslant{}& 2E[u_0] + C\Lebn{R^+}\theta 
	\Lebn{\nabla u(t)}2^{\frac{d}{\theta}}\Lebn{u_0}2^{4-\frac{d}{\theta}}
	+ \frac12 \Lebn{R^+}\infty \Lebn{u_0}2^4.
\end{align*}
This gives us $ \Lebn{\nabla u(t)}2 \leqslant C $ since $\theta>d/2$.
Otherwise, we have
\begin{multline*}
	\Lebn{\nabla u(t)}2^2
	\leqslant 2E[u_0] + \frac12
	\iint_{{\mathbb{R}}^{d+d}} V(x-y) |u(t,x)|^2 |u(t,y)|^2 dxdy\\
	+\frac12 \Lebn{ (R^+* |u|^2) |u|^2}1.
\end{multline*}
Since $u_0 \in \widetilde\Sigma^{1,1}_{V}$, 
we have $X[u(t)]=0$.
Therefore, Assumption $(\rm{V3})$ yields
\begin{align*}
	&\iint_{{\mathbb{R}}^{d+d}} V(x-y) |u(t,x)|^2 |u(t,y)|^2 dxdy\\
	&{}= \iint_{{\mathbb{R}}^{d+d}} K(x,y) |u(t,x)|^2 |u(t,y)|^2 dxdy\\
	&{}\quad + \int_{{\mathbb{R}}^d}V(x)|u(t,x)|^2dx \int_{{\mathbb{R}}^d}|u(t,y)|^2dy\\
	&{}\quad + \int_{{\mathbb{R}}^d}W(x)|u(t,x)|^2dx \cdot X[u(t)] \\
	&{}\leqslant C \Lebn{u_0}2^2 \Lebn{\Jbr{V(\cdot)}^{1/2}u(t)}2^2.
\end{align*}
Since $u$ solves \eqref{eq:gH}, plugging the estimate of Lemma \ref{lem:estVu},
one sees that
\[
	\norm{\nabla u}_{L^\infty_tL^2}^2 \leqslant C + C|t|^{\frac{2}{2-\kappa}}
	 \norm{\nabla u}_{L^\infty_tL^2}^{\frac2{2-\kappa}}
	+C\Lebn{\nabla u(t)}2^{\frac{d}{\theta}}.
\]
By Young's inequality,
\begin{multline*}
	\norm{\nabla u}_{L^\infty_tL^2}^2
	\leqslant C + \left(\frac14 \norm{\nabla u}_{L^\infty_tL^2}^2 + C_{\kappa}
	\left(C |t|^{\frac2{2-\kappa}}\right)^{\frac{2-\kappa}{1-\kappa}}\right)\\
	+\left(\frac14 \norm{\nabla u}_{L^\infty_tL^2}^2 + C_\theta \right).
\end{multline*}
Thus, we conclude that
$\norm{\nabla u}_{L^\infty_tL^2} \leqslant C \Jbr{t}^{\frac1{1-\kappa}}$.
\end{proof}

\subsection{\texorpdfstring{Proof of Theorem \ref{thm:general}}{Proof of Theorem 3.4}}
We have shown that \eqref{eq:gmH} is globally well-posed in $\widetilde\Sigma^{1,1}_{V}$.
Let us next extend the global existence of solution to \eqref{eq:gH}
for a general data, that is, for a data which \emph{does not necessarily} satisfy
\eqref{eq:neutrality}.
\begin{proof}[Proof of Theorem \ref{thm:general}]
Take a nonzero $u_0 \in \Sigma^{1,1}_V$ and define a positive constant $M$
and $d$-dimensional vectors ${\bf a}$ and $ {\bf b}$ as in \eqref{def:M} and \eqref{def:ab}, respectively.
Set $v_0(x) =(\pi_{-\bf a}\tau_{-\bf b} u_0) (x)$.
One easily verifies that $v_0$ satisfies \eqref{eq:neutrality}.
Namely, $v_0 \in \widetilde\Sigma^{1,1}_{V}$.
Then, applying Theorem \ref{thm:LWPgmH} (2) and Lemma \ref{lem:H1bound},
we obtain a global solution
$\widetilde{u}$ of \eqref{eq:gH} which conserves the mass, the energy, and the momentum.
The solution depends continuously on $v_0$, and so on $u_0$
(Recall that $e^{itA}\phi$ is continuous with respect to the parameter $M$
in $A$).
We now define a function $u$ by $u= \exp({i \frac{|{\bf a}|^2}{2}t})
	\tau_{{\bf a}t+{\bf b}} \pi_{{\bf a}}  \widetilde{u}$
as in \eqref{def:ut}.
Then, $u$ belongs to the same class as $\widetilde{u}$ and solves \eqref{eq:gH} with $u(0)=u_0$.
The solution $u$ conserves the mass because
\[
	\Lebn{u(t)}2 = \Lebn{\widetilde{u}(t)}2 = \Lebn{v_0}2 =\Lebn{u_0}2.
\]
Similarly,
\begin{align*}
	\Lebn{\nabla u(t)}2^2
	&{}=\Lebn{\nabla \widetilde{u}(t) + i {\bf a}\widetilde{u}(t)}2^2 \\
	&{}=\Lebn{\nabla \widetilde{u}(t)}2^2  - 2{\bf a}\cdot 
	P[\widetilde{u}(t)]
	+ |{\bf a}|^2 \Lebn{\widetilde{u}(t)}2^2
\end{align*}
and
\begin{align*}
	\iint_{{\mathbb{R}}^{2d}}V{(x-y)} |u(t,x)|^2|u(t,y)|^2 dxdy
	= \iint_{{\mathbb{R}}^{2d}}V{(x-y)} |\widetilde{u}(t,x)|^2|\widetilde{u}(t,y)|^2 dxdy
\end{align*}
hold. These give us the conservation of energy
\begin{multline*}
	E[u(t)] = E[\widetilde{u}(t)] -
	{\bf a}\cdot P[\widetilde{u}(t)]
	+ \frac12|{\bf a}|^2 \Lebn{\widetilde{u}(t)}2^2 \\
	=E[v_0] - {\bf a}\cdot P[v_0] + \frac12|{\bf a}|^2 \Lebn{v_0}2^2
	= E[u_0]
\end{multline*}
and the conservation of momentum
\begin{align*}
	P[u(t)]
	&{}= \operatorname{Im} \int_{{\mathbb{R}}^d}\overline{\widetilde{u}(t,y)}(\nabla \widetilde{u}(t,y) +i{\bf a}\widetilde{u}(t,y))dy \\
	&{}= P[\widetilde{u}(t)] +{\bf a}\Lebn{\widetilde{u}(t)}2^2
 =P[v_0] +{\bf a}M = P[u_0].
\end{align*}
The estimate on $\Lebn{\nabla u(t)}2$ is given in Lemma \ref{lem:H1bound},
and then the estimate on $\Lebn{\Jbr{V(\cdot)}^{1/2} u(t)}2$
follows from Lemma \ref{lem:estVu}.

So far, we have shown all the statement 
except for the uniqueness.
Let $w \in L^\infty([-T,T],\Sigma^{1,1}_V)\cap L^{8\zeta/d}([-T,T], W^{1,1}_{V,\frac{4\zeta}{2\zeta-1}})$ be another
solution of \eqref{eq:gH} with $w(0)=u_0$ which conserves $P[w(t)]$.
Then, $\partial_t w \in L^\infty((-T,T),(\Sigma^{1,1}_V)^\prime)$ holds from \eqref{eq:gH}.
We define 
$\widetilde{v} \in L^\infty([-T,T],\Sigma^{1,1}_V)\cap L^{8\zeta/d}_{\mathrm{loc}}({\mathbb{R}}, W^{1,1}_{V,\frac{4\zeta}{2\zeta-1}})$
as in \eqref{def:ut}.
Then, $\widetilde{v}$ is also a solution of \eqref{eq:gH} and conserves the momentum.
Since $X[\widetilde{v}(0)]=P[\widetilde{v}(0)]=0$,
we can see that $X[\widetilde{v}(t)] \equiv 0$.
Now, $\partial_t \widetilde{v} \in L^\infty((-T,T),(\Sigma^{1,1}_V)^\prime)$.
Multiply \eqref{eq:gH} by $\overline{\widetilde{v}}$ and integrate its imaginary part
to yield
$\Lebn{\widetilde{v}(t)}2^2=\Lebn{\widetilde{v}(0)}2^2=\Lebn{u_0}2^2$.
Then, we deduce that $\widetilde{v}$ solves \eqref{eq:gmH}.
By the uniqueness of \eqref{eq:gmH}, we obtain $\widetilde{v}=v$.
Back to the transform \eqref{def:ut}, this implies $w=u$.
\end{proof}

\subsection{\texorpdfstring{Proof of Theorem \ref{thm:main1}}{Proof of Theorem 1.1}}\label{subsec:appl}
Now, we are in a position to complete the proof of our main theorem.
Set 
\begin{equation}\label{def:VR}
	V(x)=\lambda |x|^\gamma \chi(x), \quad R(x)=\lambda|x|^\gamma (1-\chi(x)),
\end{equation}
where $\chi$ is a smooth radial non-decreasing (with respect to $|x|$)
function such that $\chi\equiv1$ for $|x|\geqslant2$
and $\chi\equiv0$ for $|x|\leqslant1$.
One immediately sees that $R \in L^\infty$ and Assumptions $(\rm{R1})$ and $(\rm{R2})$ are fulfilled
with $\zeta=\theta=\infty$. 
Thus, Theorem \ref{thm:main1} follows from Theorem \ref{thm:general}
if we prove that $V$ satisfies Assumptions
$(\rm{V1})$,  $(\rm{V2})$, $(\rm{V3})$, and $(\rm{V4})$.
We shall demonstrate merely $(\rm{V3})$ since the others are trivial.
Remark that $(\rm{V2})$ holds with $\kappa = \frac{2(\gamma-1)}{\gamma}$ and
that $\kappa<1$ if and only if $\gamma<2$.
\begin{lemma}\label{lem:esttK}
Let $\gamma \in (1,2]$ and 
\[
	\widetilde{K}(x,y) = |x-y|^{\gamma} - |x|^{\gamma}+ \gamma |x|^{\gamma -2} x\cdot y.
\]
There exists a positive constant 
$C$ depending only on $\gamma $ such that
\[
	\sup_{x \in {\mathbb{R}}^d} |\widetilde{K}(x,y)|\leqslant C |y|^{\gamma }.
\]
\end{lemma}
\begin{proof}
The case $y=0$ is trivial. We hence fix $ {\mathbb{R}}^d \ni y \neq 0$. 
It immediately follows that
\[
	\sup_{x, |x|\leqslant 2|y|} |\widetilde{K}(x,y)| \leqslant 3^\gamma|y|^\gamma + 2^\gamma |y|^\gamma + \gamma 2^{\gamma-1}|y|^{\gamma}
	\leqslant C |y|^\gamma.
\]
We now consider the case $|x| \geqslant 2|y|$. 
An elementary computation shows that $K$ is written as
\begin{align*}
	\widetilde{K}(x,y) ={}& \int_0^1 \frac{\partial}{\partial a}|x-ay|^\gamma da + \gamma |x|^{\gamma -2} x\cdot y \\
	={}& -\int_0^1 \gamma |x-ay|^{\gamma-2} y\cdot(x-ay) da + \gamma |x|^{\gamma -2} x\cdot y \\
	={}& \gamma |y|^2 \int_0^1 a |x-ay|^{\gamma-2} da 
	-  \gamma x\cdot y\int_0^1 \int_0^1\frac{\partial}{\partial b}|x-bay|^{\gamma-2}  db da \\
	={}& \gamma |y|^2 \int_0^1 a |x-ay|^{\gamma-2} da \\
	&{} +  \gamma(\gamma-4)  x\cdot y\int_0^1 \int_0^1 |x-bay|^{\gamma-4}ay\cdot(x-bay)   db da\\
	=:{}& \widetilde{K}_1(x,y) + \widetilde{K}_2(x,y).
\end{align*}
For any integer $m\geqslant2$, it holds that
\[
	\sup_{x, m|y| \leqslant |x| \leqslant (m+1)|y|}|\widetilde{K}_1(x,y)|
	\leqslant \gamma |y|^2 \int_0^1 a ((m-1)|y|)^{\gamma-2}da=\frac{\gamma}{2(m-1)^{2-\gamma}}|y|^\gamma.
\]
Therefore,
\[
	\sup_{x,|x|\geqslant2|y|}|\widetilde{K}_1(x,y)| \leqslant \sup_{m\geqslant 2}\frac{\gamma}{2(m-1)^{2-\gamma}}|y|^\gamma
	=\frac{\gamma}{2}|y|^\gamma
\]
Similarly, we have
\begin{align*}
	&\sup_{x, m|y| \leqslant |x| \leqslant (m+1)|y|}|\widetilde{K}_2(x,y)|\\
	&{}\leqslant \gamma(4-\gamma)  (m+1)|y|^2\int_0^1 \int_0^1 ((m-1)|y|)^{\gamma-4}a((m+1)|y|^2)   db da \\
	&{}=\frac{\gamma(4-\gamma)(m+1)^2}{2(m-1)^{4-\gamma}}|y|^\gamma.
\end{align*}
Since $\sup_{m\geqslant 2}(m+1)^2(m-1)^{4-\gamma}=3^2$, we conclude that
\[
	\sup_{x,|x|\geqslant2|y|}|\widetilde{K}_2(x,y)|
	\leqslant \sup_{m\geqslant 2}\frac{\gamma(4-\gamma)(m+1)^2}{2(m-1)^{4-\gamma}}|y|^\gamma
	=\frac{9\gamma(4-\gamma)}{2}|y|^\gamma,
\]
which completes the proof.
\end{proof}

\begin{proposition}
Let $V$ be as in \eqref{def:VR}.
If $\gamma \in (1,2)$ then $V$ satisfies Assumption $(\rm{V3})$
with $W(x)=-\lambda\gamma x \Jbr{x}^{\gamma-2}$.
More precisely, if we put
\[
	K(x,y) =\lambda \chi(|x-y|)|x-y|^{\gamma} - \lambda\chi(|x|)|x|^{\gamma}+\lambda \gamma \Jbr{x}^{\gamma -2} x\cdot y,
\]
then there exists a positive constant 
$C$ depending only on $\gamma $ and $\lambda$ such that
\[
	\sup_{x \in {\mathbb{R}}^d} |K(x,y)|\leqslant C \Jbr{y}^{\gamma }, \quad
	\sup_{x \in {\mathbb{R}}^d} |\nabla_x K(x,y)|\leqslant C \Jbr{y}.
\]
\end{proposition}
\begin{proof}
For simplicity, let $\lambda=1$.
Let $\widetilde{K}$ be as in Lemma \ref{lem:esttK}.
We deduce that
\[
	\sup_{x \in {\mathbb{R}}^d}|\widetilde{K}(x,y)-K(x,y)|\leqslant
	2^{1+\gamma}
	+ \sup_{x \in {\mathbb{R}}^d}|\gamma (|x|^{\gamma -2} - \Jbr{x}^{\gamma -2}) x\cdot y|.
\]
Let us estimate the second term of the right hand side.
An elementary calculation shows
\[
	|x|^\nu -\Jbr{x}^\nu
	= -\int_0^1 \partial_a (a+|x|^2)^{\frac{\nu}2} da
	= -\frac{\nu}{2}\int_0^1 (a+|x|^2)^{\frac{\nu}2-1} da
\]
for any $\nu$, and so 
\[
	\sup_{|x|\geqslant 1} |\gamma (|x|^{\gamma -2} - \Jbr{x}^{\gamma -2}) x\cdot y|
	\leqslant  \frac{\gamma(2-\gamma)}{2} |y|
	\leqslant C\Jbr{y}.
\]
It is obvious that
\[
	\sup_{|x|\leqslant 1} |
	\gamma (|x|^{\gamma -2} - \Jbr{x}^{\gamma -2}) x\cdot y|
	\leqslant C\Jbr{y}.
\]
Then, the first inequality follows from Lemma \ref{lem:esttK}.

Let us proceed to the second inequality.
Notice that
\begin{multline*}
	\nabla_x K(x,y)
	= \gamma|x-y|^{\gamma-2}(x-y)\chi(|x-y|) +|x-y|^\gamma \nabla \chi(|x-y|)
	\\
	 -\gamma|x|^{\gamma-2}x\chi(|x|)- |x|^\gamma \nabla \chi(|x|)\\
	+\gamma\Jbr{x}^{\gamma-2}y +\gamma(\gamma-2)\Jbr{x}^{\gamma-4}(x\cdot y)x.
\end{multline*}
Since $\nabla \chi(|x|) =0$ for $|x|\leqslant1$ and $|x|\geqslant2$, it holds that 
\[
	\sup_{x \in {\mathbb{R}}^d} ||x-y|^\gamma \nabla \chi(|x-y|)|
	+ ||x|^\gamma \nabla \chi(|x|)| \leqslant 2^{\gamma+1}\Lebn{\nabla \chi}\infty .
\]
We also deduce  that 
\[
	\sup_{x \in {\mathbb{R}}^d}
	|\Jbr{x}^{\gamma-2}y|+ |\Jbr{x}^{\gamma-4}(x\cdot y)x|
	\leqslant C|y|
\]
for $\gamma\in (1,2) $.
Now, It holds that
\begin{align*}
	&|x-y|^{\gamma-2}(x-y)\chi(|x-y|) - |x|^{\gamma-2}x\chi(|x|) \\
	&{}= \int_0^1 \partial_a |x-ay|^{\gamma-2}(x-ay)\chi(|x-ay|) da\\
	&{}= (2-\gamma)\int_0^1 |x-ay|^{\gamma-4}y\cdot(x-ay)(x-ay)\chi(|x-ay|) da \\
	&\quad - \int_0^1 |x-ay|^{\gamma-2}y\chi(|x-ay|) da \\
	&\quad - \int_0^1 |x-ay|^{\gamma-3}y\cdot(x-ay) \chi^\prime(|x-ay|) da.
\end{align*}
Notice that $\sup_{x\in {\mathbb{R}}^d}\abs{|x|^{\gamma-2}\chi(|x|)}\leqslant 1$.
Hence, the first term of the right hand side of above equality is estimated as
\begin{align*}
	&\sup_{x\in {\mathbb{R}}^d}\abs{\int_0^1 |x-ay|^{\gamma-4}y\cdot(x-ay)(x-ay)\chi(|x-ay|) da}\\
	&\quad \leqslant |y| \sup_{0\leqslant a\leqslant1}\sup_{x\in {\mathbb{R}}^d}|x-ay|^{\gamma-2} \chi(|x-ay|)  \leqslant |y|.
\end{align*}
One can obtain similar estimates for other terms. Thus, we conclude that
\begin{equation*}
	\sup_{x \in {\mathbb{R}}^d}
	\abs{|x-y|^{\gamma-2}(x-y)\chi(|x-y|) - |x|^{\gamma-2}x\chi(|x|)} \leqslant C|y|.
\end{equation*}
\end{proof}

\appendix
\section{\texorpdfstring{Global well-posedness of \eqref{eq:nH} for $\gamma\in(0,1]$}
{Global well-poseness of (nH) for gamma in (0,1]}}
Here, we adapt the abstract theory established in Section 3
to the case $\gamma \in (0,1]$, 
which is characterized as the case where Assumption $(\rm{V3})$ is satisfied with $W=0$.
In such a special setting, results in Section 3 become better.
This is because we do not need the conservation of momentum any longer.
As a result, Assumption $(\rm{V4})$ can be removed and the uniqueness holds 
in the class in which the solution lies (without the conservation of momentum).
\begin{theorem}\label{thm:appendix1}
	Let $d\geqslant 1$, $\gamma \in (0,1]$, and $\lambda \in {\mathbb{R}}$.
	Then \eqref{eq:nH} is globally well-posed in $\Sigma^{1,\gamma/2}$.
	Moreover, the uniqueness holds unconditionally.
\end{theorem}
This theorem is an immediate consequence of the following.
\begin{theorem}\label{thm:appendix2}
	Let Assumptions $(\rm{V1})$, $(\rm{V2})$, $(\rm{R1})$, and $(\rm{R2})$ be satisfied.
	Also suppose that $(\rm{V3})$ holds with $W=0$.
	Then, \eqref{eq:gH} is globally well-posed in $\Sigma^{1,1}_V$.
\end{theorem}
\begin{remark}\label{rmk:appendixunique}
In Theorem \ref{thm:appendix2}, the uniqueness holds in
$u \in C({\mathbb{R}},\Sigma^{1,1}_V) \cap L^{8\zeta/d}_{\mathrm{loc}}({\mathbb{R}},W^{1,1}_{V,\frac{4\zeta}{2\zeta-1}})$,
where $\zeta$ is the number defined in Assumption $(\rm{R1})$.
If we can choose $\zeta=\infty$, then the uniqueness holds unconditionally.
\end{remark}
\begin{proof}[Proof of Theorem \ref{thm:appendix2}]
We choose $W=0$ and let $u$ be the unique local solution of \eqref{eq:gmH}
given in Theorem \ref{thm:LWPgmH} (1).
By the conservation of mass, the right hand side of \eqref{eq:gmH}
becomes
\[
	-u \left( (V*|u|^2)- V(x) \Lebn{u_0}2^2\right) 
	= - (V*|u|^2)u +M V(x) u.
\]
Hence $u$ solves \eqref{eq:gH}.
Energy conservation follows from this fact as in the proof of Theorem
\ref{thm:LWPgmH} (2).

By Lemma \ref{lem:estVu} and energy conservation, we have
\[
	\norm{\nabla u}_{L^\infty_tL^2}^2 \leqslant C + C|t|^{\frac2{2-\kappa}}
	 \norm{\nabla u}_{L^\infty_tL^2}^{\frac2{2-\kappa}}
	 + C\norm{\nabla u}_{L^\infty_tL^2}^{\frac{d}{\theta}}
\]
as in the proof of Lemma \ref{lem:H1bound}.
This yields $\norm{\nabla u}_{L^\infty_tL^2} \leqslant C \Jbr{t}^{\frac1{1-\kappa}}$.
Hence, again by Lemma \ref{lem:estVu}, one sees that $\norm{u(t)}_{\Sigma^{1,1}_V}$
never blows up in finite time.

We shall prove the uniqueness.
Let $v \in C({\mathbb{R}};\Sigma^{1,1}_V)$ be another solution of \eqref{eq:gH}.
One verifies that $v$ conserves mass and so that $v$ solves \eqref{eq:gmH}.
Thus, we conclude from the uniqueness of \eqref{eq:gmH} (given in Theorem
\ref{thm:LWPgmH} (1)) that $u=v$ follows.
\end{proof}
\begin{proof}[Proof of Theorem \ref{thm:appendix1}]
Take a non-decreasing function $\chi$ so that 
$\chi(r)= 0$ for $r\leqslant 1$ and $\chi(r)=1$ for $r\geqslant2$.
We put $V(x)=\lambda|x|^{\gamma}\chi(|x|)$ and $R(x)=\lambda|x|^\gamma(1-\chi(|x|))$.
It is obvious that $R\in L^\infty$ satisfies $(\rm{R1})$ and $(\rm{R2})$ with $\zeta=\theta=\infty$ and
that $V$ satisfies Assumption $(\rm{V1})$.
We infer that $(\rm{V2})$ is fulfilled with $\kappa=0$.
Furthermore, $(\rm{V3})$ follows with $W=0$ from the estimates 
$ \norm{V(x-y)-|x-y|^\gamma}_{L^\infty_x}\leqslant2^\gamma$,
$\Lebn{V(x) - |x|^\gamma}\infty\leqslant 2^\gamma$, and
$ \norm{|x-y|^\gamma-|x|^\gamma}_{L^\infty_x}\leqslant |y|^\gamma$.
Hence, Theorem \ref{thm:appendix1} follows from Theorem \ref{thm:appendix2}.
Uniqueness holds unconditionally since we can chose $\zeta=\infty$ (see Remark \ref{rmk:appendixunique}).
\end{proof}

We can also obtain results on 
\begin{equation}\label{eq:logH}
	\left\{
	\begin{aligned}
	&i \partial_t u + \frac12 \Delta u = -\lambda (\log|x|*|u|^2) u, \\
	&u(0)=u_0,
	\end{aligned}
	\right.
\end{equation}
where $(t,x) \in {\mathbb{R}}^{1+d}$.
\begin{theorem}\label{thm:appendix3}
	Let $d\geqslant 1$ and $\lambda \in {\mathbb{R}}$.
	Then \eqref{eq:logH} is globally well-posed in 
$\{f \in H^1 ;\ \Jbr{\log\Jbr{x}}^{1/2}f \in L^2 \}$.
\end{theorem}
This reproduce \cite[Theorem 1.1]{Ma2DSPe} when $d=2$.
The proof is similar to that of Theorem \ref{thm:appendix1}.
We choose $V(x)=\lambda\chi(|x|)\log|x|$ and $W(x)=\lambda(1- \chi(|x|))\log|x|$.

\subsection*{Acknowledgments}
This research is supported by Japan Society for the Promotion of Science(JSPS)
Grant-in-Aid for Young Scientists (B) 24740108.

\providecommand{\bysame}{\leavevmode\hbox to3em{\hrulefill}\thinspace}

\providecommand{\href}[2]{#2}

\end{document}